\documentclass[reqno]{amsart}
\usepackage{amsmath,amsthm,amssymb,mathrsfs,stmaryrd,eucal,mathbbol}
\usepackage[all,cmtip]{xy}
\usepackage{hyperref}
\usepackage{graphics}
\usepackage{epstopdf}

\newcommand\void[1]       {}

\newcommand{\be}{\begin{equation}}
\newcommand{\ee}{\end{equation}}
\newcommand{\bnu}{\begin{enumerate}}
\newcommand{\enu}{\end{enumerate}}

\numberwithin{equation}{section}

\newcommand{\CA}{\mathcal{A}}
\newcommand{\CB}{\mathcal{B}}
\newcommand{\CC}{\mathcal{C}}
\newcommand{\CD}{\mathcal{D}}
\newcommand{\CE}{\mathcal{E}}

\newcommand{\CM}{\mathcal{M}}
\newcommand{\CN}{\mathcal{N}}

\newcommand{\CP}{\mathcal{P}}

\newcommand{\CX}{\mathcal{X}}

\newcommand{\FD}{\mathfrak{D}}

\newcommand{\FZ}{\mathfrak{Z}}

\newcommand{\bk}{\mathbf{k}}

 \DeclareMathOperator{\Hom}{Hom}

 \DeclareMathOperator{\Id}{Id}

 \DeclareMathOperator{\Fun}{Fun}
 
 \DeclareMathOperator{\funend}{\mathcal{E}nd}

 \DeclareMathOperator{\Alg}{Alg}

 \DeclareMathOperator{\LMod}{LMod}
 \DeclareMathOperator{\RMod}{RMod}
 \DeclareMathOperator{\BMod}{BMod}

\newcommand{\adj}[4]{\xymatrix{ #1 \ar@<.5ex>[r]^-{#3} & #2 \ar@<.5ex>[l]^-{#4}}}

\newcommand{\nn}{\nonumber \\}
\newcommand{\one}{\mathbf{1}}
\newcommand{\rev}{\mathrm{rev}}
\newcommand{\op}{\mathrm{op}}

\newcommand{\mtc}{\mathcal{MT}\mathrm{en}_k^{\mathrm{ind}}}
\newcommand{\bmtc}{\mathcal{BT}\mathrm{en}_k}
\newcommand{\mfus}{\mathcal{MF}\mathrm{us}_k^{\mathrm{ind}}}
\newcommand{\bfus}{\mathcal{BF}\mathrm{us}_k^{\mathrm{cl}}}
\newcommand{\ph}{\mathrm{ph}}

\newtheorem{thm}{Theorem}[subsection]
\newtheorem{lem}[thm]{Lemma}
\newtheorem{prop}[thm]{Proposition}
\newtheorem{cor}[thm]{Corollary}

\newtheorem{prop-defn}[thm]{Proposition-Definition}

\theoremstyle{definition}
\newtheorem{defn}[thm]{Definition}

\newtheorem{exam}[thm]{Example}
\newtheorem{rem}[thm]{Remark}

\theoremstyle{remark}

\begin{document}

\title{The center functor is fully faithful}
\maketitle

\begin{center}
{\large
Liang Kong$^{a}$,\,
Hao Zheng$^{b}$\,
~\footnote{Emails:
{\tt  kongl@sustc.edu.cn, hzheng@math.pku.edu.cn}}}
\\[1em]
$^a$ Shenzhen Institute for Quantum Science and Engineering, \\
and Department of Physics, \\
Southern University of Science and Technology, Shenzhen 518055, China
\\[0.5em]
$^b$ Department of Mathematics, Peking University,\\
Beijing, 100871, China
\end{center}

\vspace{0.2cm}
\begin{abstract}
We prove that the notion of Drinfeld center defines a functor from the category of indecomposable multi-tensor categories with morphisms given by bimodules to that of braided tensor categories with morphisms given by monoidal bimodules. Moreover, we apply some ideas from the physics of topological orders to prove that the center functor restricted to indecomposable multi-fusion categories (with additional conditions on the target category) is fully faithful. As byproducts, we provide new proofs to some important known results in fusion categories. In physics, this fully faithful functor gives the precise mathematical description of the boundary-bulk relation for 2+1D anomaly-free topological orders with gapped boundaries.


\end{abstract}

\tableofcontents

\newpage
\section{Introduction}

It is well known in classical algebra that the notion of center is not functorial unless we consider only simple algebras. Recent developments in the mathematical theory of 2d rational conformal field theories \cite{dkr1,dkr2} and that of topological orders in condensed matter physics \cite{kitaev-kong,kong-icmp,kong-wen-zheng}, however, suggest that this rather trivial functoriality of the center for simple algebras becomes highly non-trivial for the categorical analogue of simple algebras, if we also modify the morphisms in the target category properly.

In this paper, we prove that the notion of center (or Drinfeld center) defines a symmetric monoidal functor $\FZ$ from the category $\mtc$ of indecomposable multi-tensor categories over a field $k$ (see Definition \ref{def:k-linear-monoidal}) with morphisms given by bimodules to the category $\bmtc$ of braided tensor categories with morphisms given by monoidal bimodules (see Definition \ref{def:monoidal-modules}). More precisely, for $\CC, \CD\in \mtc$ and a $\CC$-$\CD$-bimodule $\CM$, the center functor $\FZ$ is defined by $\CC \mapsto \FZ(\CC)$ and $\CM\mapsto \FZ(\CM)$, where $\FZ(\CM):=\Fun_{\CC|\CD}(\CM,\CM)$ is the category of $\CC$-$\CD$-bimodule functors.

A multi-tensor category can be regarded as the categorical analogue of a semisimple algebra as illustrated in the following dictionary (see Definition\,\ref{def:k-linear-monoidal}).
\begin{center}
\begin{tabular}{|c||c|}
\hline
finite monoidal category & finite-dimensional algebra \\
\hline
multi-tensor category & semisimple algebra \\
\hline indecomposable multi-tensor category & simple algebra \\
\hline
\end{tabular}
\end{center}
For our purpose, we need to develop the theory of multi-tensor categories in many directions.
In particular, we will study the theory of the tensor product over a monoidal category in details in Section \ref{sec:tensor-prod}-\ref{sec:mtc} and that of the tensor product over a braided multi-tensor category in Section \ref{sec:relative-over-BTC}-\ref{sec:rigidity}. After all these preparations, in Section \ref{sec:fun-center}, we prove that $\FZ: \mtc\to\bmtc$ is a well-defined symmetric monoidal functor.

\void{
In particular, we generalize some facts in classical algebras to finite monoidal categories. Let $A$, $B$ be finite-dimensional algebras over $k$ and $M$ a finite-dimensional $A$-$B$-bimodule. We have the following facts in classical algebras:
\bnu
\item If $A,B$ are semisimple, then $M$ has both a left dual and a right dual.
\item If $M$ is invertible, then the action of $A,B$ induces isomorphisms of algebras $A\simeq \Hom_{B^\rev}(M,M)$, $B^\rev\simeq \Hom_A(M,M)$, where $B^\rev$ is the opposite algebra of $B$. Moreover, $A$ and $B$ share the same center.
\item Conversely, if $A$ and $B$ are semisimple and the action of $A,B$ induces isomorphisms $A\simeq \Hom_{B^\rev}(M,M)$ and $B^\rev\simeq \Hom_A(M,M)$, then $M$ is invertible.
\item If $A$ is simple, then every algebra homomorphism $f: A \to B$ splits as
$A \xrightarrow{\Id_A \otimes 1} A \otimes_{Z(A)} \Hom_{A|B}(B,B) \simeq B$.
\enu
We recall their analogues for (rigid) finite monoidal categories in Section \ref{sec:morita} and in Theorem \ref{thm:fun=ph-mor}.
}

For applications of this center functor in topological orders in condensed matter physics \cite{lw-mod,kitaev-kong,kong-wen-zheng},
one needs to consider semisimple multi-tensor categories, also called multi-fusion categories \cite{eno2005}. Let $\mfus$ be the subcategory of $\mtc$ consisting of indecomposable multi-fusion categories and semisimple bimodules and $\bfus$ the subcategory of $\bmtc$ consisting of nondegenerate braided fusion categories and closed multi-fusion bimodules (see Definition \ref{def:monoidal-modules} and Definition \ref{def:mf-bimodule}).  In Section \ref{sec:fully-faithful}, we prove that the center functor $\FZ: \mfus \to \bfus$ is fully faithful. Our proof of the fully-faithfulness is inspired by the physical intuition of the boundary-bulk relation in topological orders \cite{kong-wen-zheng}. Mathematically, it amounts to the following logical steps: (1) for a closed multi-fusion $\FZ(\CC)$-$\FZ(\CD)$-bimodule $\CE$, we show that $\CC\boxtimes_{\FZ(\CC)} \CE \boxtimes_{\FZ(\CD)} \CD^\rev$ is a multi-fusion category with a trivial center; (2)
a multi-fusion category with a trivial center is equivalent to the category $\Fun_\bk(\CM,\CM)$, where $\bk$ is the category of finite dimensional vector spaces over a field $k$ and $\CM$ is certain semisimple finite module category over $\bk$ (see Definition\,\ref{def:finite-module}); (3) $\CE$ determines a unique $\CC$-$\CD$-bimodule structure on $\CM$; (4) $\CE \mapsto \CM$ gives the inverse map of $\FZ$ on morphisms. This completes the proof.

The fully-faithfulness of $\FZ$ immediately implies two important known results in fusion categories: (1) two fusion categories $\CC$ and $\CD$ are Morita equivalent if and only if $\FZ(\CC)\simeq^{br} \FZ(\CD)$ \cite{mueger0,eno2008}; (2) there is a group isomorphism between the group of the equivalence classes of semisimple invertible $\CC$-$\CC$-bimodules and the group of the equivalence classes of braided auto-equivalences of $\FZ(\CC)$ \cite{eno2009}. As byproducts, our proof of the fully-faithfulness of $\FZ$ provides new proofs to the above two results and also slightly generalize them.

\void{
\begin{figure}[tb]
 \begin{picture}(150, 100)
   \put(20,10){\scalebox{2}{\includegraphics{pic-lw-defect-duality.pdf}}}
   \put(-75,-55){
     \setlength{\unitlength}{.75pt}\put(-18,-19){
     \put(125, 108)       { $\CC$}
     \put(295, 108)       { $\CD$}
     \put(213,116)      { $ a $}
     \put(160, 160)    { $\FZ(\CC)$}
     \put(240, 160)    { $\FZ(\CD)$}
     \put(210,215)     { $\CE$}
     }\setlength{\unitlength}{1pt}}
  \end{picture}
\caption{The 1d topological orders on the boundary are given by fusion categories $\CC$ and $\CD$, and their 2d bulk phases are given by $\FZ(\CC)$ and $\FZ(\CD)$. The closed domain wall between $\FZ(\CC)$ and $\FZ(\CD)$ is given by a monoidal $\CC$-$\CD$-bimodule $\CE$ (see Definition \ref{def:monoidal-modules}). The defect junction of $\CC,\CD,\CE$ is a single particle-like topological defect which is given by an object $a$ in a $\CC$-$\CD$-bimodule $\CM$.
}
\label{fig:defect-duality}
\end{figure}
}

Besides our main results, a few notable new results are the explicit equivalences between different realizations of the tensor product over a monoidal category (Theorem \ref{thm:tensor-prod}, Corollary \ref{cor:tensor-prod}, Proposition \ref{prop:tensor-prod2}),
the criteria of rigidity (Proposition \ref{prop:rigidity}, Corollary \ref{cor:rigidity}), an alternative and equivalent definition of a monoidal functor (Theorem \ref{thm:fun=ph-mor}) and Theorem \ref{thm:closed-bimodule}.

In physics, this fully faithful center functor $\FZ$ gives the precise mathematical description of the boundary-bulk relation for anomaly-free 2+1D topological orders with gapped boundaries \cite{kong-wen-zheng}.

\medskip
\noindent {\bf Acknowledgement}: LK was supported by NSFC under Grant No. 11071134 and by the Center of Mathematical Sciences and Applications at Harvard University, and by the start fund from Yau Mathematical Science Center at Tsinghua University, and is currently supported by the Science, Technology and Innovation Commission of Shenzhen Municipality (Grant Nos. ZDSYS20170303165926217 and JCYJ20170412152620376) and Guangdong Innovative and Entrepreneurial Research Team Program (Grant No. 2016ZT06D348). HZ is supported by NSFC under Grant No.11131008.

\section{Preliminaries}

Given left modules $\CM,\CN$ over a monoidal category $\CC$, we use $\Fun_\CC(\CM,\CN)$ to denote the category of $\CC$-module functors that {\it preserve finite colimits} throughout this paper. We remind readers that a functor between abelian categories preserves finite colimits if and only if it is right exact.

\subsection{Rigid monoidal categories}
In this subsection, we recall some basic facts about the rigidity of a monoidal category and set our notations.

\smallskip
Let $\CC$ be a monoidal category. We say that an object $a\in\CC$ is {\it left dual} to an object $b\in\CC$ and $b$ is {\it right dual} to $a$, if there are a {\it unit}, i.e. a morphism $u:\one\to b\otimes a$, and a {\it counit}, i.e. a morphism $v:a\otimes b\to\one$, such that the composed morphisms
$$a\simeq a\otimes\one \xrightarrow{\Id_a\otimes u} a\otimes b\otimes a \xrightarrow{v\otimes\Id_a} \one\otimes a \simeq a$$
$$b\simeq \one\otimes b \xrightarrow{u\otimes\Id_b} b\otimes a\otimes b \xrightarrow{\Id_b\otimes v} b\otimes\one \simeq b$$
are identity morphisms. We also denote $a=b^L$ and $b=a^R$.

We say that $\CC$ is {\it rigid}, if every object has both a left dual and a right dual. In this case, taking the left dual determines an equivalence
$$\delta^L:\CC\to\CC^\op, \quad\quad a \mapsto a^L,$$
and taking the right dual determines an equivalence
$$\delta^R:\CC\to\CC^\op, \quad\quad a \mapsto a^R,$$

\void{
\begin{rem}
Our convention of the left/right duals is compatible with that of the left/right adjoint functors. More precisely, the assignment $a \mapsto a\otimes -$ defines a monoidal functor from $\CC$ to the category $\funend(\CC)$ of endo-functors. The functor $a\otimes - \in \funend(\CC)$ is left adjoint to $a^R \otimes -$ and right adjoint to $a^L\otimes -$.
\end{rem}

If $\CC$ is rigid, taking the left dual, i.e. $a \mapsto a^L$, determines a functor $\delta^L:\CC\to\CC^\op$ which is unique up to unique isomorphism. Similarly, the operation of taking right dual determines a functor $\delta^R:\CC\to\CC^\op$ which is also unique up to unique isomorphism. Since both $b^L\otimes a^L$ and $(a\otimes b)^L$ represent the functor $\Hom_\CC(\one,a\otimes b\otimes-)$, $b^L\otimes a^L$ is canonically isomorphic to $(a\otimes b)^L$, which promotes $\delta^L$ to a monoidal equivalence $\CC^\rev\simeq\CC^\op$. Similarly, $\delta^R$ gives another monoidal equivalence $\CC^\rev \simeq \CC^\op$.
}

\begin{rem} \label{rem:rigid-(co)limits}
Let $\CC$ be a rigid monoidal category and $(\CM, \odot)$ (or $\CM$) a left $\CC$-module \cite{ostrik}. Then the functor $a\odot-:\CM\to\CM$ is left adjoint to $a^R\odot-$ and right adjoint to $a^L\odot-$ for $a\in\CC$. Therefore, the action $\odot:\CC\times\CM\to\CM$ preserves all limits and colimits in the second variable. In particular, the tensor product $\otimes:\CC\times\CC\to\CC$ preserves all limits and colimits separately in each variable.
\end{rem}

Given a left module $\CM$ over a rigid monoidal category $\CC$, the opposite category $\CM^\op$ admits two natural right $\CC$-module structures, which are denoted by $\CM^{\op|L}$ and $\CM^{\op|R}$, respectively. More precisely, for $x\in \CM$, $a\in \CC$, we have
\begin{align}
(\CM^{\op|L}, \odot^L): \quad\quad & x\odot^L a:=a^L\odot x,  \nn
(\CM^{\op|R}, \odot^R): \quad\quad & x\odot^R a := a^R \odot x. \nonumber
\end{align}

\void{
Similarly, a right $\CC$-module $\CN$ is automatically a left $\CC^\rev$-module. The category $\CN^\op$ is automatically a right $\CC^\op$-module. If $\CC$ is rigid, there are two canonical left $\CC$-module structures $(\CN^{L|\op},\odot^L)$ and $(\CN^{R|\op},\odot^R)$ induced from the monoidal equivalences $\delta^L$ and $\delta^R$, respectively. More precisely, we have
\begin{align}
(\CN^{L|\op}, \odot^L): \quad\quad & a\odot^L y := y \odot a^L, \nn
(\CM^{R|\op}, \odot^R): \quad\quad & a\odot^R y := y \odot a^R \nonumber
\end{align}
for $y\in \CN, a\in \CC$. The right $\CC$-module structures on both $(\CN^{L|\op})^{\op|R}$ and $(\CN^{R|\op})^{\op|L}$ agree with that on $\CN$. Similarly, for a left $\CC$-module $\CM$, the left $\CC$-module structures on $(\CM^{\op|L})^{R|\op}$ and $(\CM^{\op|R})^{L|\op}$ coincide with that on $\CM$.

Let $\CC$ and $\CD$ be monoidal categories and $\CP$ a $\CC$-$\CD$-bimodule.
There are four natural $\CD$-$\CC$-bimodule structures on $\CP^\op$ defined as follows:
\begin{align}
\CP^{L|\op|L}: \, d\odot^L p \odot^L c := c^L \odot p \odot d^L, \quad\quad
\CP^{L|\op|R}: \,  d\odot^L p \odot^R c := c^R \odot p \odot d^L,  \nn
\CP^{R|\op|L}: \, d\odot^R p \odot^L c := c^L \odot p \odot d^R, \quad\quad
\CP^{R|\op|R}: \,  d\odot^R p \odot^R c := c^R \odot p \odot d^R.\nonumber
\end{align}
}

Let $\CC$ be a monoidal category. We use $\Alg(\CC)$ to denote the category of algebras in $\CC$. Given algebras $M,N$ in $\CC$, we use $\LMod_M(\CC)$, $\RMod_N(\CC)$ and $\BMod_{M|N}(\CC)$ to denote the category of left $M$-modules, right $N$-modules and $M$-$N$-bimodules in $\CC$, respectively. Note that $\LMod_N(\CC)$ is automatically a right $\CC$-module and that $\RMod_N(\CC)$ is a left $\CC$-module.

\begin{rem} \label{rem:lm-mod}
Let $\CC$ be a rigid monoidal category, and let $M$ be an algebra in $\CC$. Given a right $M$-module $V\in\RMod_M(\CC)$, the action $V\otimes M\to V$ gives arise to a morphism $M\otimes V^R\to V^R$ which endows $V^R$ with the structure of a left $M$-module. Therefore, the functor $\delta^R$ induces an equivalence $\RMod_M(\CC)\simeq\LMod_M(\CC)^{R|\op}$ of left $\CC$-modules.

Similarly, given a right $M$-module $V\in\RMod_M(\CC)$, the action $V\otimes M\to V$ induces a morphism $M^{LL} \otimes V^L \to V^L$ which endows $V^L$ with the structure of a left $M^{LL}$-module. Therefore, the functor $\delta^L$ induces an equivalence $\RMod_M(\CC)\simeq\LMod_{M^{LL}}(\CC)^{L|\op}$ of left $\CC$-modules.
\end{rem}

\void{
\begin{rem} \label{rem:adjoint}
Let $\CC$ be a rigid monoidal category and $F:\CM\to\CN$ a module functor between two left $\CC$-modules. If the underlying functor $F$ has a right adjoint $G:\CN\to\CM$, then $G$ automatically admits the structure of a module functor. Actually, we have $\Hom_\CM(x,G(a\otimes y)) \simeq \Hom_\CN(F(x),a\otimes y) \simeq \Hom_\CN(a^L\otimes F(x),y) \simeq \Hom_\CN(F(a^L\otimes x),y) \simeq \Hom_\CM(a^L\otimes x,G(y)) \simeq \Hom_\CM(x,a\otimes G(y))$ which determines the desired isomorphism $G(a\otimes y)\simeq a\otimes G(y)$. It follows that the monoidal category $\Fun_\CC(\CM,\CM)$ is rigid if and only if every functor $F\in\Fun_\CC(\CM,\CM)$ admits both a left adjoint and a right adjoint in $\Fun(\CM,\CM)$.
\end{rem}
}

\begin{defn}
Let $\CC$ be a monoidal category and $\CM$ a left $\CC$-module. Given objects $x,y\in\CM$, we define an object $[x,y]_\CC\in\CC$ (or $[x,y]$ for simplicity), if exists, by the mapping property
$$
\Hom_\CC(a,[x,y])\simeq\Hom_\CM(a\odot x,y),
$$
and refer to it as the {\it internal hom} between $x$ and $y$. We say that $\CM$ is {\it enriched in $\CC$}, if $[x,y]$ exists for every pair of objects $x,y\in\CM$.
\end{defn}

\begin{rem} \label{rem:inhom}
If $\CM$ is enriched in $\CC$ and if $a,b\in\CC$ have left duals, then we have a canonical isomorphism for $x,y\in\CM$
\be \label{eq:a[xy]b}
a\otimes[x,y]\otimes b^L\simeq[b\odot x,a\odot y].
\ee
Indeed, we have
$\Hom_\CC(c,a\otimes[x,y]\otimes b^L)
\simeq \Hom_\CC(a^L\otimes c\otimes b,[x,y])
\simeq \Hom_\CM(a^L\odot c\odot b\odot x,y)
\simeq \Hom_\CM(c\odot b\odot x,a\odot y)
\simeq \Hom_\CC(c,[b\odot x,a\odot y])
$ for $c\in\CC$.
\end{rem}

\begin{rem}
The identity morphism $\Id_x:x\to x$ induces a morphism $\one\to[x,x]$. Moreover, the natural map $\Hom_\CC(a,[x,y])\times\Hom_\CC(b,[y,z]) \simeq \Hom_\CM(a\odot x,y)\times\Hom_\CM(b\odot y,z) \to \Hom_\CM(b\odot a\odot x,z) \simeq \Hom_\CC(b\otimes a,[x,z])$, determines a canonical morphism $[y,z]\otimes[x,y]\to[x,z]$ if we set $a=[x,y]$ and $b=[y,z]$. It follows that $[x,x]$ is an algebra in $\CC$ and $[x,y]$ is a right module over $[x,x]$ (see for example \cite{ostrik}).
\end{rem}


\begin{lem} \label{lem:int-hom}
Let $\CC$ be a rigid monoidal category that admits coequalizers, and $\CM=\RMod_M(\CC)$ for some algebra $M\in\Alg(\CC)$. Then $\CM$ is enriched in $\CC$ and we have $[x,y] \simeq (x\otimes_M y^R)^L$ for $x,y\in\CM$.
\end{lem}

\begin{proof}
We have $\Hom_\CC(a,(x\otimes_M y^R)^L) \simeq \Hom_\CC(a\otimes x\otimes_M y^R,\one) \simeq \Hom_\CM(a\otimes x,y)$, where we have implicitly used the fact that the tensor functor $\otimes$ preserves colimits (recall Remark \ref{rem:rigid-(co)limits}).
\end{proof}

\begin{thm} \label{thm:reconst}
Let $\CC$ be a rigid monoidal category that admits coequalizers, and let $\CM$ be a left $\CC$-module that admits coequalizers. Then $\CM\simeq\RMod_M(\CC)$ for some algebra $M\in\Alg(\CC)$ if and only if the following conditions are satisfied.
\begin{enumerate}
\item $\CM$ is enriched in $\CC$.
\item There is an object $P\in\CM$ such that the functor $[P,-]:\CM\to\CC$ is conservative and preserves coequalizers.
\end{enumerate}
In this case, the functor $[P,-]$ induces an equivalence $\CM\simeq\RMod_{[P,P]}(\CC)$.
\end{thm}

\begin{proof}
Necessity. $(1)$ is given by Lemma \ref{lem:int-hom}. The object $M$ satisfies $(2)$ because $[M,-]$ is just the forgetful functor, which is obviously conservative and preserves colimits.

Sufficiency. By definition, the functor $F=-\odot P:\CC\to\CM$ is left adjoint to the functor $G=[P,-]:\CM\to\CC$.
Applying the Barr-Beck theorem, we see that the functor $G$ exhibits $\CM$ monadic over $\CC$.
Since $\CC$ is rigid, we have $G\circ F = [P,-\odot P] \simeq - \otimes [P,P]$ by Remark \ref{rem:inhom}. Therefore, we conclude that $\CM\simeq\RMod_{[P,P]}(\CC)$.
\end{proof}

\void{
\begin{rem}
When $\CC$ is the category of finite-dimensional vector spaces over a field, Condition $(2)$ in the above theorem means that there is a projective object $P\in\CM$ such that $\Hom_\CM(P,x)\simeq0$ only if $x\simeq0$.
\end{rem}
}

\begin{rem}\label{rem:reconst}
From the proof we see that the sufficiency of Theorem \ref{thm:reconst} still holds if we drop the rigidity of $\CC$ but require in Condition $(2)$ that the canonical morphism $-\otimes[P,P] \to [P,-\odot P]$ is an isomorphism.
\end{rem}

\begin{rem} \label{rem:adjoint}
Let $\CC$ be a rigid monoidal category and $F:\CM\to\CN$ a $\CC$-module functor between two left $\CC$-modules. If the underlying functor $F$ has a right adjoint $G:\CN\to\CM$, then $G$ automatically admits the structure of a module functor. Actually, we have $\Hom_\CM(x,G(a\otimes y)) \simeq \Hom_\CN(F(x),a\otimes y) \simeq \Hom_\CN(a^L\otimes F(x),y) \simeq \Hom_\CN(F(a^L\otimes x),y) \simeq \Hom_\CM(a^L\otimes x,G(y)) \simeq \Hom_\CM(x,a\otimes G(y))$ which determines the desired isomorphism $G(a\otimes y)\simeq a\otimes G(y)$. It follows that the monoidal category $\Fun_\CC(\CM,\CM)$ is rigid if and only if every functor $F\in\Fun_\CC(\CM,\CM)$ admits both a left adjoint and a right adjoint in $\Fun(\CM,\CM)$.
\end{rem}

\subsection{Tensor product of module categories} \label{sec:tensor-prod}

The finite-colimit-preserving condition in the following definition is inspired by Deligne-Kelly tensor product \cite{Kelly,deligne} and Lurie's tensor product of presentable $\infty$-categories \cite{lurie2} (see similar definitions in \cite{tam,eno2009,bbj}).

\begin{defn}
Let $\CC$ be a monoidal category, $\CM$ a right $\CC$-module and $\CN$ a left $\CC$-module. Suppose that the categories $\CC,\CM,\CN$ admit finite colimits and that the functors $\otimes:\CC\times\CC\to\CC$, $\odot:\CM\times\CC\to\CM$ and $\odot:\CC\times\CN\to\CN$ preserve finite colimits separately in each variable.

A {\it balanced $\CC$-module functor} is a functor $F:\CM\times\CN\to\CD$ equipped with an isomorphism $F\circ(\odot\times\Id_\CN) \simeq F\circ(\Id_\CM\times\odot):\CM\times\CC\times\CN\to\CD$ satisfying the following conditions:
\begin{itemize}
\item $F$ preserves finite colimits separately in each variable.
\item For $a,b\in\CC$, $x\in\CM$, $y\in\CN$, the evident diagrams
$$
\xymatrix{
  & F(x,y) \ar[ld]_\simeq \ar[rd]^\simeq \\
  F(x\odot\one,y) \ar[rr]^\simeq & & F(x,\one\odot y) \\
}
$$
$$
\xymatrix{
  F((x\odot a)\odot b,y) \ar[r]^\simeq \ar[d]^\simeq & F(x\odot a,b\odot y) \ar[r]^\simeq & F(x,a\odot(b\odot y)) \ar[d]^\simeq \\
  F(x\odot(a\otimes b),y) \ar[rr]^\simeq && F(x,(a\otimes b)\odot y) \\
}
$$
are commutative.
\end{itemize}
We use $\Fun^{bal}_\CC(\CM,\CN;\CD)$ to denote the category of balanced $\CC$-module functors $F:\CM\times\CN\to\CD$.

The {\it tensor product} of $\CM$ and $\CN$ over $\CC$ is a category $\CM\boxtimes_\CC\CN$, which admits all finite colimits, together with a balanced $\CC$-module functor $\boxtimes_\CC:\CM\times\CN\to\CM\boxtimes_\CC\CN$, such that, for every category $\CD$, which admits all finite colimits, composition with $\boxtimes_\CC$ induces an equivalence $\Fun(\CM\boxtimes_\CC\CN,\CD)\simeq\Fun^{bal}_\CC(\CM,\CN;\CD)$.
\end{defn}

\begin{rem}
It is worthwhile to remind readers the usual universal property of the tensor product as illustrated in the following commutative diagram:
$$
\xymatrix{
\CM \times \CN \ar[r]^{\boxtimes_\CC}  \ar[rd]_F & \CM \boxtimes_\CC \CN  \ar[d]^{\exists !\, \underline{F} }  \\
& \CD
}
$$
for $F\in \Fun^{bal}_\CC(\CM,\CN;\CD)$.
\end{rem}

\begin{thm} \label{thm:tensor-prod}
Let $\CC$ be a monoidal category such that $\CC$ admits finite colimits and the tensor product $\otimes:\CC\times\CC\to\CC$ preserves finite colimits separately in each variable. Let $\CM=\LMod_M(\CC)$, $\CM'=\RMod_M(\CC)$ and $\CN=\RMod_N(\CC)$ for some algebras $M,N$ in $\CC$. We have the following assertions:
\begin{enumerate}
\item The balanced $\CC$-module functor $\CM\times\CN \to \BMod_{M|N}(\CC)$ defined by $(x,y)\mapsto x\otimes y$ exhibits $\BMod_{M|N}(\CC)$ as the tensor product $\CM\boxtimes_\CC\CN$.

\item The balanced $\CC$-module functor $\CM \times \CN \to \Fun_\CC(\CM',\CN)$ defined by $(x,y) \mapsto -\otimes_M x\otimes y$ exhibits $\Fun_\CC(\CM',\CN)$ as the tensor product $\CM\boxtimes_\CC\CN$.

\end{enumerate}
\end{thm}
\begin{proof}
Since any functor $F\in \Fun_\CC(\CM',\CN)$ preserves coequalizers, we have canonical isomorphisms $F(-)=F(-\otimes_M M) \simeq - \otimes_M F(M)$, where $F(M)\in \BMod_{M|N}(\CC)$. As a consequence, we have a canonical equivalence $\BMod_{M|N}(\CC)\simeq \Fun_\CC(\CM',\CN)$ defined by $z \mapsto -\otimes_M z$. Then it is clear that
(2) follows from (1). It remains to prove (1).

Our assumptions on $\CC$ and $\otimes$ imply that $\BMod_{M|N}(\CC)$ admits finite colimits. Let $F:\CM\times\CN\to\CD$ be a balanced $\CC$-module functor where $\CD$ admits finite colimits. We construct a functor $G: \BMod_{M|N}(\CC) \to \CD$ as follows. Given a bimodule $V\in\BMod_{M|N}(\CC)$, note that $V\simeq V\otimes_N N$ is the coequalizer of the evident diagram $V\otimes N\otimes N\rightrightarrows V\otimes N$. We set $G(V)$ to be the coequalizer of the evident diagram $F(V,N\otimes N)\rightrightarrows F(V,N)$. Given a bimodule map $V\to V'$, we have an induced morphism $G(V)\to G(V')$ such that $G$ is a functor. By the construction of $G$ and our assumptions on $\otimes$ and $F$, we obtain that $G$ preserves finite colimits. Thus we obtain a functor $\Fun^{bal}_\CC(\CM,\CN;\CD) \to \Fun(\BMod_{M|N}(\CC),\CD)$.

It remains to identify $F(x,y)$ with $G(x\otimes y)$ for $x\in\CM$ and $y\in\CN$. Actually, by our assumption on $F$, $G(x\otimes y)$ can be identified with the coequalizer of the evident diagram $F(x,y\otimes N\otimes N)\rightrightarrows F(x,y\otimes N)$, which is nothing but $F(x,y)$.
\end{proof}

\begin{rem} \label{rem:tensor-product-exact}
When $\CC$ is rigid, the tensor product functor $\CM \times \CN \to \BMod_{M|N}(\CC)$ defined by $(x,y)\mapsto x\otimes y$ preserves limits and colimits in each variable.
\end{rem}

\begin{rem}
Results similar to Theorem \ref{thm:tensor-prod} have appeared earlier in various frameworks (see for example \cite{eno2009,dss,lurie2,bbj,dn}).
\end{rem}

\begin{cor} \label{cor:tensor-prod}
Let $\CC$ be a rigid monoidal category which admits finite colimits. Assume $\CM=\RMod_M(\CC)$, $\CN=\RMod_N(\CC)$ for some algebras $M,N$ in $\CC$.
\begin{enumerate}
\item The balanced $\CC$-module functor $\CM^{\op|L}\times\CN \to \Fun_\CC(\CM,\CN)$ defined by
$$
(x,y)\mapsto [-,x]^R\odot y
$$
exhibits $\Fun_\CC(\CM,\CN)$ as the tensor product $\CM^{\op|L}\boxtimes_\CC\CN$.

\item We have a natural isomorphism $$\Hom_{\CM^{\op|L}\boxtimes_\CC\CN}(x\boxtimes_\CC y,x'\boxtimes_\CC y') \simeq \Hom_\CC(\one,[y,y']\otimes[x',x])$$ for $x,x'\in\CM^\op$ and $y,y'\in\CN$.

\item The formula $x\boxtimes_\CC y\mapsto y\boxtimes_\CC x$ determines an equivalence $\CM^{\op|R}\boxtimes_\CC\CN \simeq (\CN^{\op|L}\boxtimes_\CC\CM)^\op$.

\end{enumerate}
\end{cor}

\begin{proof}
$(1)$ By Remark \ref{rem:lm-mod}, $x\mapsto x^R$ defines a canonical equivalence of right $\CC$-modules $\CM^{\op|L}=\RMod_M(\CC)^{\op|L} \simeq \LMod_M(\CC)$. Using Theorem \ref{thm:tensor-prod} and Lemma \ref{lem:int-hom}, we obtain the following canonical equivalences:
\begin{align}
\CM^{\op|L}\boxtimes_\CC\CN &\simeq \BMod_{M|N}(\CC) \simeq \Fun_\CC(\CM, \CN) \nn
x\boxtimes_\CC y &\mapsto x^R\otimes y  \mapsto -\otimes_M x^R \otimes y=[-,x]^R \otimes y.
\label{eq:can-iso-tensor-prod}
\end{align}

$(2)$ We have $\Hom_{\CM^\op\boxtimes_\CC\CN}(x\boxtimes_\CC y,x'\boxtimes_\CC y') \simeq \Hom_{\BMod_{M|N}(\CC)}(x^R\otimes y,x'^R\otimes y') \simeq \Hom_\CC(x' \otimes_M x^R \otimes y \otimes_N y'^R, \one) \simeq \Hom_\CC(\one,[y,y']\otimes[x',x])$.

$(3)$ By Remark \ref{rem:lm-mod}, the functor $x\mapsto x^L$ defines the canonical equivalence $\CM^{\op|R} \simeq \LMod_{M^{LL}}(\CC)$. We obtain the following canonical equivalences:
$$
\CM^{\op|R}\boxtimes_\CC\CN \simeq \BMod_{M^{LL}|N}(\CC) \xrightarrow{\delta^R}
\BMod_{N|M}(\CC)^\op \simeq (\CN^{\op|L}\boxtimes_\CC\CM)^\op
$$
defined by
$x\boxtimes_\CC y \mapsto x^L \otimes y \mapsto (x^L\otimes y)^R = y^R \otimes x \mapsto y\boxtimes_\CC x.  \nonumber$ 
\end{proof}

\void{
In the setting of Corollary \ref{cor:tensor-prod}, $\CM^{\op|L}\boxtimes_\CC \CN$ has a natural structure of a $\Fun_\CC(\CN,\CN)$-$\Fun_\CC(\CM,\CM)$-bimodule, defined by
\be \label{eq:fxyg}
f\odot (x\boxtimes_\CC y) \odot^R g= g^R(x)\boxtimes_\CC f(y)
\ee for $f\in \Fun_\CC(\CN,\CN)$ and $g\in \Fun_\CC(\CM,\CM)$. So does the finite category $\Fun_\CC(\CM, \CN)$. To make the right $\Fun_\CC(\CM,\CM)$-action on $\CM^{\op|L}$ explicit, we denote it by $\CM^{R|\op|L}$.

\begin{prop} \label{prop:bimodule-eq}
In the setting of Corollary \ref{cor:tensor-prod}, the canonical equivalence $\CM^{R|\op|L} \boxtimes_\CC \CN \xrightarrow{\simeq} \Fun_\CC(\CM, \CN)$, defined by $x\boxtimes_\CC y \mapsto [-,x]^R\odot y$,
is an equivalence between $\Fun_\CC(\CN,\CN)$-$\Fun_\CC(\CM,\CM)$-bimodules.
\end{prop}
\begin{proof}
It follows from the following identities: $g^R(x) \boxtimes_\CC f(y) \mapsto
[-,g^R(x)]^R\odot f(y) \simeq f([-,g^R(x)]^R \odot y) \simeq f([g(-),x]^R\odot y) \simeq
f\circ ([-,x]^R\odot y) \circ g$.
\end{proof}
}

\begin{prop} \label{prop:tensor-prod2}
Let $\CC$ be a rigid monoidal category which admits finite colimits. Assume $\CM=\RMod_M(\CC)$, $\CN=\RMod_N(\CC)$ for some algebras $M,N$ in $\CC$.
If taking right adjoints defines an equivalence between $\Fun_\CC(\CN, \CM)^\op$ and $\Fun_\CC(\CM, \CN)$, then the balanced $\CC$-module functor $\CM^{\op|R}\times\CN \to \Fun_\CC(\CM,\CN)$ defined by $(x,y)\mapsto [x,-]\odot y$ exhibits $\Fun_\CC(\CM,\CN)$ as the tensor product $\CM^{\op|R}\boxtimes_\CC\CN$.
\end{prop}

\begin{proof}
By (3) in Corollary \ref{cor:tensor-prod}, we have
$$\CM^{\op|R}\boxtimes_\CC\CN \simeq (\CN^{\op|L}\boxtimes_\CC\CM)^\op \simeq \Fun_\CC(\CN,\CM)^\op \simeq \Fun_\CC(\CM,\CN)
$$
where the third equivalence is defined by taking right adjoint. Note that these equivalences map an object $x\boxtimes_\CC y \in \CM^{\op|R}\boxtimes_\CC\CN$ as follows:
$$
x\boxtimes_\CC y \mapsto y\boxtimes_\CC x \mapsto [-,y]^R\odot x \mapsto ([-,y]^R\odot x)^R
$$
where $([-,y]^R\odot x)^R$ is the right adjoint of the functor $[-,y]^R\odot x\in \Fun_\CC(\CN, \CM)$.

It remains to prove that
\begin{equation}  \label{eq:right-adjoint-fun}
([-,y]^R\odot x)^R \simeq [x,-]\odot y.
\end{equation}
Without loss of generality, we may assume $\CM=\RMod_M(\CC)$ and $\CN=\RMod_N(\CC)$ for some algebras $M,N\in \Alg(\CC)$. For $m\in \CM, n\in \CN$, we have the following canonical isomorphisms
\begin{align*}
&\Hom_\CM([n,y]^R\odot x, m)
\simeq \Hom_\CM(n\otimes_N y^R\otimes x, m) \\
&\simeq \Hom_\CC(n\otimes_N y^R \otimes x \otimes_M m^R, \one)
\simeq \Hom_\CN(n, (y^R\otimes x \otimes_Mm^R)^L) \\
&\simeq \Hom_\CN(n, (x \otimes_Mm^R)^L\otimes y)
\simeq \Hom_\CN(n, [x,m]\odot y).
\end{align*}
where we have used Lemma \ref{lem:int-hom} in the first and the last step. This shows that $[x,-]\odot y$ is right adjoint to $[-,y]^R\odot x$, as desired.
\end{proof}

\void{
\begin{rem}
Note that the functor $\CM^{\op|R} \boxtimes_\CC \CN \xrightarrow{\simeq} \Fun_\CC(\CM, \CN)$ defined by $x\boxtimes_\CC y \mapsto [x,-]\odot y$ also intertwines the following actions:
$$
f\odot (x\boxtimes_\CC y) \odot^R y \mapsto f \circ ([x,-]\odot y) \circ g^{RR};
$$
and
the functor $\CM^{\op|L} \boxtimes_\CC \CN \xrightarrow{\simeq} \Fun_\CC(\CM, \CN)$ defined by $x\boxtimes_\CC y \mapsto [-,x]^R\odot y$ intertwines the following actions:
$$
f\odot (x\boxtimes_\CC y) \odot^L y \mapsto f \circ ([-,x]^R\odot y) \circ g^{LL}.
$$
\end{rem}
}


\subsection{Multi-tensor categories} \label{sec:mtc}

Let $k$ be a field. We denote by $\bk$ the category of finite dimensional vector spaces over $k$. We will denote $\CM\boxtimes_\bk\CN$ simply by $\CM\boxtimes\CN$ for $\bk$-modules $\CM,\CN$ when $k$ is clear from the context.

\begin{defn}
A {\it $\bk$-linear category} is a $k$-linear category that admits finite colimits.
By a {\it $\bk$-linear functor} between two $\bk$-linear categories we mean a $k$-linear functor that preserves finite colimits.
\end{defn}

\begin{defn}
By a {\it finite category} over $k$ we mean a $\bk$-linear category $\CC$ which is equivalent to $\RMod_A(\bk)$ for some finite dimensional $k$-algebra $A$.
\end{defn}

\begin{rem}
An intrinsic description of a finite category is a $k$-linear abelian category such that all morphism spaces are finite dimensional, every object has finite length, and it has finitely many simple objects, each of which has a projective cover (see \cite{egno}). 
\end{rem}

\begin{defn}  \label{def:k-linear-monoidal}
A {\it finite monoidal category} over $k$ is a monoidal category $\CC$ such that $\CC$ is a finite category over $k$ and the tensor product $\otimes:\CC\times\CC\to\CC$ is $k$-bilinear on morphisms and right exact separately in each variable. 
A {\it multi-tensor category} is a rigid finite monoidal category.
A {\it tensor category} is a multi-tensor category with a simple tensor unit.
We say that a multi-tensor category is {\it indecomposable} if it is neither zero nor the direct sum of two nonzero multi-tensor categories.
\end{defn}

\begin{rem} {\rm
In \cite{egno}, a multi-tensor category is called {\it a finite multi-tensor category}. We omit ``finite'' for simplicity. In \cite[Definition\,4.1.1]{egno}, $k$ is assumed to be algebraically closed, in this case, the notion of a tensor category defined above coincides with that of a finite tensor category in \cite{egno}.
}
\end{rem}

\begin{defn} \label{def:finite-module}
Given a finite monoidal category $\CC$, we say that a left $\CC$-module $\CM$ is {\it finite} if $\CM$ is a finite category and the action $\CC\times\CM\to\CM$ is $k$-bilinear on morphisms and right exact separately in each variable. The notions of a {\it finite right module} and a {\it finite bimodule} are defined similarly.
\end{defn}

\begin{lem} \label{lem:M-enrich-in-C}
Let $\CC$ be a finite monoidal category and $\CM$ a finite left $\CC$-module. Then $\CM$ is enriched in $\CC$.
\end{lem}

\begin{proof}
The functor $\Hom_\CM(-\odot x,y)^\vee:\CC\to\bk$, where ${}^\vee$ represents the dual vector space, is right exact for $x,y\in\CM$, hence, a functor in $\Fun_\bk(\CC,\bk)$. By Corollary \ref{cor:tensor-prod}(1), the assignment $a\mapsto\Hom_\CC(-,a)^\vee$ defines an equivalence $\CC^\op \simeq \Fun_\bk(\CC,\bk)$. Therefore, $\Hom_\CM(-\odot x,y)^\vee$ is representable. Namely, $[x,y]$ exists.
\end{proof}

\begin{rem}
Actually, one can show that a $k$-linear functor between two finite categories has a right (resp. left) adjoint if and only if it is right (resp. left) exact.
\end{rem}

\begin{prop} \label{prop:mod-alg}
Let $\CC$ be a multi-tensor category and $\CM$ a finite left $\CC$-module. We have $\CM\simeq\RMod_M(\CC)$ for some algebra $M$ in $\CC$.
\end{prop}

\begin{proof}
Assume $\CM=\RMod_A(\bk)$ and $\CC=\RMod_B(\bk)$. According to Lemma \ref{lem:M-enrich-in-C}, $\CM$ is enriched in $\CC$. In particular, the functor $[A,-]:\CM\to\CC$ is well defined. The functor $\Hom_\CC(\one,[A,-]) \simeq \Hom_\CM(A,-)\simeq \Hom_\bk(k,-)$ is conservative, thus $[A,-]$ is also conservative. Let $F: \CC\to \bk$ be the forgetful functor. The functor $F([A,-]) \simeq \Hom_\CC(B,[A,-]) \simeq \Hom_\CM(A,B^R\odot-)\simeq \Hom_\bk(k, B^R\odot-)$ is exact, hence $[A,-]$ is exact. Applying Theorem \ref{thm:reconst}, we establish the proposition.
\end{proof}

\begin{lem} \label{lem:mod-alg}
Let $\CC$ be a finite monoidal category and $M,N$ algebras in $\CC$. Then $\BMod_{M|N}(\CC)$ is a finite category. In particular, $\LMod_M(\CC)$ and $\RMod_N(\CC)$ are finite categories.
\end{lem}

\begin{proof}
Assume $\CC=\RMod_A(\bk)$. The functor $\Hom_{\BMod_{M|N}(\CC)}(M\otimes A\otimes N,-) \simeq \Hom_\CC(A,-) \simeq \Hom_\bk(k,-)$ is exact and conservative. Then apply Theorem \ref{thm:reconst}.
\end{proof}

\begin{cor}\label{cor:prod-well-define}
Let $\CC$ be a multi-tensor category, $\CM$ and $\CN$ finite right and left $\CC$-modules, respectively. Then $\CM\boxtimes_\CC\CN$ is a finite category.
\end{cor}

\begin{proof}
Combine Theorem \ref{thm:tensor-prod}, Proposition \ref{prop:mod-alg} and Lemma \ref{lem:mod-alg}.
\end{proof}

\begin{rem}
Results similar to Proposition \ref{prop:mod-alg} and Corollary \ref{cor:prod-well-define} were proved in \cite{eo} for $\CM, \CN$ being exact modules.
\end{rem}

\begin{rem}
The universal property of the tensor product $\boxtimes_\CC$ can be enhanced to the $\bk$-linear setting as illustrated in the following commutative diagram:
$$
\xymatrix{
\CM \boxtimes \CN \ar[r]^{\boxtimes_\CC}  \ar[rd]_F & \CM \boxtimes_\CC \CN  \ar[d]^{\exists !\, \underline{F} }  \\
& \CD
}
$$
for $F\in \Fun^{bal}_\CC(\CM,\CN;\CD)$.
\end{rem}

The part (1) of the following proposition was proved in \cite[Proposition\,3.8]{dss}. For readers convenience, we briefly sketch the proof.
\begin{prop} \label{prop:prod-fun}
Let $\CC,\CC'$ be finite monoidal categories and $\CM=\RMod_M(\CC)$, $\CN=\RMod_N(\CC)$, $\CM'=\RMod_{M'}(\CC')$, $\CN'=\RMod_{N'}(\CC')$ for some algebras $M,N\in\Alg(\CC)$, $M',N'\in\Alg(\CC')$.

$(1)$ The formula $(V,V')\mapsto V\boxtimes V'$ determines an equivalence
$$\CM\boxtimes\CM' \simeq \RMod_{M\boxtimes M'}(\CC\boxtimes\CC').$$

$(2)$ The formula $(\phi,\psi)\mapsto\phi\boxtimes\psi$ determines an equivalence $$\Fun_\CC(\CM,\CN)\boxtimes\Fun_{\CC'}(\CM',\CN') \simeq \Fun_{\CC\boxtimes\CC'}(\CM\boxtimes\CM',\CN\boxtimes\CN').$$
\end{prop}

\begin{proof}
$(1)$ In view of Corollary \ref{cor:prod-well-define}, the categories $\CM \boxtimes \CM'$, $\CC\boxtimes\CM'$ are well-defined. Let $G: \RMod_{M\boxtimes M'}(\CC\boxtimes\CC') \to \CC\boxtimes\CC'$ be the forgetful functor. Let $G': \CM\boxtimes\CM' \to \CC\boxtimes\CM'$ and $G'': \CC\boxtimes\CM' \to \CC\boxtimes\CC'$ be the functors canonically induced from the forgetful functors, and let $F,F',F''$ be their left adjoints, respectively. $G$ is exact and conservative. One can show that $G',G''$ are also exact and conservative. Applying the Barr-Beck theorem, we see that the functors $G$ and $G''\circ G'$ exhibit $\RMod_{M\boxtimes M'}(\CC\boxtimes\CC')$ and $\CM\boxtimes\CM'$ monadic over $\CC\boxtimes\CC'$, respectively. Then we obtain the desired equivalence from the isomorphism of monads $G\circ F \simeq (G''\circ G')\circ(F'\circ F'')$.

$(2)$ By using the same argument, we deduce an equivalence $\BMod_{M|N}(\CC) \boxtimes \BMod_{M'|N'}(\CC') \simeq \BMod_{M\boxtimes M'|N\boxtimes N'}(\CC\boxtimes\CC')$.
\end{proof}

\subsection{The center of a monoidal category}

Recall that the {\it center} (or {\it Drinfeld center} or {\it monoidal center}) of a monoidal category $\CC$, denoted by $\FZ(\CC)$, is the category of pairs $(z,\beta_{z,-})$, where $z \in \CC$ and $\beta_{z,-}: z\otimes - \to -\otimes z$ is a half braiding (see for example \cite{js,majid}).
Morphisms in $\FZ(\CC)$ are morphisms between the first components preserving the half-braidings.
The category $\FZ(\CC)$ has a natural structure of a braided monoidal category with the braidings defined by the half-braidings \cite{js2}.
Moreover, $\FZ(\CC)$ can be naturally identified with the category of $\CC$-$\CC$-bimodule functors from $\CC$ to $\CC$.

\void{
Sometimes it is useful to characterize $\FZ(\CC)$ by the following universal property of center. The tensor product of $\CC$ induces a unital monoidal action $\rho:  \CC \times \FZ(\CC)\to \CC$, i.e. $\rho$ is a monoidal functor such that the tensor unit of $\FZ(\CC)$ acts by identity on $\CC$. Given any monoidal category $\CX$ and any unital monoidal action $f: \CC \times \CX \to \CC$, there exists a unique (up to isomorphism) monoidal functor $\underline{f}: \CX \to \FZ(\CC)$ such that the following diagram
\be \label{diag:univ-property}
\xymatrix{
& \CC \times \FZ(\CC)  \ar[dr]^\rho & \\
\CC \times \CX \ar[rr]^f \ar[ur]^{\exists ! \, \Id_\CC\times \underline{f}} & & \CC
}
\ee
is commutative up to isomorphism.
}

\begin{rem} \label{rem:rev-center}
For a monoidal category $\CC$, we use $\CC^\rev$ to denote the monoidal category which has the same underlying category as $\CC$ but equipped with the reversed tensor product $a\otimes^\rev b := b\otimes a$.
For a braided monoidal category $\CC$ with a braiding $\beta_{a,b}: a\otimes b \xrightarrow{\simeq} b\otimes a$ for $a,b\in \CC$, we use $\overline\CC$ to denote the same monoidal category $\CC$ but equipped with the anti-braiding $\bar\beta_{a,b}:=(\beta_{b,a})^{-1}$.
Note that a half-braiding in a monoidal category $\CC$ is identical to the inverse of a half-braiding in $\CC^\rev$. Consequently, we may simply identify
$$\FZ(\CC^\rev)=\overline{\FZ(\CC)}.$$
\end{rem}

\begin{rem}\label{rem:rigid-center}
If $\CC$ is a rigid monoidal category, then $\FZ(\CC)$ is also rigid. Actually, $\Fun_\CC(\CC,\CC)\simeq\CC^\rev$ is rigid, hence closed under taking left and right adjoints. It follows that $\FZ(\CC)\simeq\Fun_{\CC|\CC}(\CC,\CC)$ is rigid (recall Remark\,\ref{rem:adjoint}). If $\CC$ is multi-tensor, so is $\FZ(\CC)$, and the forgetful functor $\FZ(\CC) \to \CC$ is exact. 
\end{rem}

\begin{defn}
Let $\CC$ and $\CD$ be two finite monoidal categories and $\CM$ a finite $\CC$-$\CD$-bimodule. We say that $\CM$ is {\it invertible} if there is a finite $\CD$-$\CC$-bimodule $\CN$ and equivalences $\CM\boxtimes_\CD\CN\simeq\CC$ as $\CC$-$\CC$-bimodules and $\CN\boxtimes_\CC\CM\simeq\CD$ as $\CD$-$\CD$-bimodules. If such an invertible $\CC$-$\CD$-bimodule exists, $\CC$ and $\CD$ are said to be {\it Morita equivalent}.
\end{defn}

The following result was proved in various contexts with different assumptions (see for example \cite{schauenburg,mueger0,eno2009}). For reader's convenience, we sketch a proof.
\begin{prop} \label{prop:inv-bimod-center}
Let $\CC,\CD$ be finite monoidal categories and $\CM$ an invertible $\CC$-$\CD$-bimodule. The evident monoidal functors $\FZ(\CC) \to \Fun_{\CC|\CD}(\CM,\CM) \leftarrow \FZ(\CD)$ are equivalences. Moreover, the induced equivalence $\FZ(\CC)\simeq\FZ(\CD)$ preserves braidings.
\end{prop}
\begin{proof}
Let ${}_\CD\CN_\CC$ be an inverse of $\CM$. The evident functor $\CC \to \Fun_{\CD^\rev}(\CM,\CM)$ is a monoidal equivalence with its quasi-inverse given by the composed equivalence $\Fun_{\CD^\rev}(\CM,\CM) \simeq \Fun_{\CC^\rev}(\CM\boxtimes_\CD \CN, \CM\boxtimes_\CD\CN) \simeq \Fun_{\CC^\rev}(\CC,\CC)\simeq \CC$. It follows immediately that the evident monoidal functor $\FZ(\CC) \to \Fun_{\CC|\CD}(\CM,\CM)$ is an equivalence. The proof of the evident monoidal functor $\Fun_{\CC|\CD}(\CM,\CM) \leftarrow \FZ(\CD)$ being an equivalence is similar.

Suppose the induced equivalence $\FZ(\CC) \to \FZ(\CD)$ carries $c,c'$ to $d,d'$, respectively. We have the following commutative diagram for $x\in \CM$:
$$
\xymatrix{
  c\odot (c'\odot x) \ar[r]^\sim \ar[d]_{\beta_{c,c'}\odot\Id_x} & (c'\odot x)\odot d \ar[r]^\sim \ar[d]^\simeq & (x\odot d')\odot d \ar[d]^{\Id_x\odot\beta_{d',d}} \\
  c'\odot (c\odot x) \ar[r]^\sim & c'\odot (x\odot d) \ar[r]^\sim & (x\odot d)\odot d'\, . \\
}
$$
The commutativity of the left square is due to the fact that the isomorphism $c\odot - \simeq -\odot d$ is a left $\CC$-module functor; that of the right square is due to the fact that $c'\odot - \simeq -\odot d'$ is a right $\CD$-module functor. Then the commutativity of the outer square says that the induced equivalence $\FZ(\CC) \to \FZ(\CD)$ preserves braidings.
\end{proof}

\begin{lem}\label{lem:cc-inhom}
Let $\CC$ be a multi-tensor category. Regard $\CC$ as a left $\CC\boxtimes\CC^\rev$-module. We have a canonical isomorphism
$$[a\otimes b,c\otimes d] \simeq (c\boxtimes d)\otimes[\one,\one]\otimes(a^L\boxtimes b^R)$$
in $\CC\boxtimes\CC^\rev$ for $a,b,c,d\in\CC$.
\end{lem}

\begin{proof}
This follows immediately form Remark \ref{rem:inhom} because $a\boxtimes b$ and $c\boxtimes d$ admit left duals in $\CC\boxtimes\CC^\rev$.
\end{proof}

\begin{lem}\label{lem:cc-alg}
Let $\CC$ be a multi-tensor category. Then $\CC\simeq\RMod_{[\one,\one]}(\CC\boxtimes\CC^\rev)$.
\end{lem}

\begin{proof}
Note that the functor $[\one,-]:\CC\to\CC\boxtimes\CC^\rev$ is exact and conservative. Moreover, $[\one,-\odot\one] \simeq -\otimes[\one,\one]$ by Lemma \ref{lem:cc-inhom}. Then apply Theorem \ref{thm:reconst} and Remark \ref{rem:reconst}.
\end{proof}

\begin{prop} \label{prop:center-prod}
Let $\CC$ and $\CD$ be multi-tensor categories. We have
\bnu
\item $\FZ(\CC \bigoplus \CD) \simeq \FZ(\CC) \bigoplus \FZ(\CD)$.
\item The evident embeddings $\FZ(\CC),\FZ(\CD)\hookrightarrow\FZ(\CC\boxtimes\CD)$ induces an equivalence of braided monoidal categories $\FZ(\CC)\boxtimes\FZ(\CD) \simeq \FZ(\CC\boxtimes\CD)$.
\enu
\end{prop}

\begin{proof}
(1) is obvious. (2) follows from Proposition \ref{prop:prod-fun} and Lemma
\ref{lem:cc-alg}
(recall that $\FZ(\CC)\simeq\Fun_{\CC\boxtimes\CC^\rev}(\CC,\CC)$).
\end{proof}

\begin{rem}
In view of Lemma \ref{lem:cc-inhom}, there is a canonical isomorphism $\psi_c: (\one\boxtimes c)\otimes [\one,\one] \simeq [\one,c] \simeq (c\boxtimes \one)\otimes [\one,\one]$ for $c\in\CC$. Note that $\psi_c$ induces an isomorphism $([\one,\one]\odot a) \otimes c \to c\otimes ([\one,\one]\odot a)$ for $a\in \CC$, thus equips $[\one,\one]\odot a$ with a half-braiding, so as to promote $[\one,\one]\odot a$ to an object in $\FZ(\CC)$.
\end{rem}

\void{
By the universal property of $[\one,a]\odot \one \to a$, we see that $\psi_a$ is the unique isomorphism rendering the induced diagram
$$
\xymatrix{
  & a  \\
  a\otimes ([\one,\one]\odot\one) \ar[rr] \ar[ru]^{\Id_a\otimes v} && ([\one,\one]\odot\one)\otimes a \ar[lu]_{v\otimes\Id_a} \\
}
$$
commutative. It follows that $\psi$ is compatible with the monoidal structure of $\CC$, so that $\psi$ equips $[\one,\one]\odot a$ with a half braiding for every $a\in\CC$.
}

\begin{prop} \label{prop:center-adjoint}
Let $\CC$ be a multi-tensor category. The forgetful functor $\FZ(\CC)\to\CC$ is right adjoint to $a \mapsto [\one,\one]\odot a$.
\end{prop}

\begin{proof}
Let $u:\one\boxtimes\one\to [\one,\one]$ be the unit of the algebra, and let $v:[\one,\one]\odot\one\to\one$ be the canonical morphism.
The unit map $a \simeq (\one\boxtimes\one)\odot a \xrightarrow{u\odot\Id_a} [\one,\one]\odot a$, $a\in\CC$
and the counit map $[\one,\one]\odot b \xrightarrow{\beta_{-,b}\otimes\Id} b\otimes([\one,\one]\odot\one) \xrightarrow{\Id_b\otimes v} b$, $b\in\FZ(\CC)$
exhibit the forgetful functor $\FZ(\CC)\to\CC$ right adjoint to $a \mapsto [\one,\one]\odot a$.
\end{proof}

\begin{prop} \label{prop:center-cc}
Let $\CC$ be a multi-tensor category. There is a canonical isomorphism $([a,b]_{\FZ(\CC)}\odot c)^L \simeq [a^L,c^L]_{\CC\boxtimes\CC^\rev}\odot b^L$ for $a,b,c\in\CC$.
\end{prop}

\begin{proof}
We have $\Hom_{\FZ(\CC)}([a,b]_{\FZ(\CC)}^L,x) \simeq \Hom_{\FZ(\CC)}(x^R,[a,b]_{\FZ(\CC)}) \simeq \Hom_\CC(a\otimes x^R,b) \simeq \Hom_\CC(b^L\otimes a,x)$ for $x\in\FZ(\CC)$. So, $[a,b]_{\FZ(\CC)}^L \simeq [\one,\one]_{\CC\boxtimes\CC^\rev}\odot(b^L\otimes a)$ by Proposition \ref{prop:center-adjoint}. It follows that
$([a,b]_{\FZ(\CC)}\odot c)^L
\simeq ([\one,\one]_{\CC\boxtimes\CC^\rev}\odot(b^L\otimes a)) \otimes c^L
\simeq ((\one\boxtimes c^L)\otimes[\one,\one]_{\CC\boxtimes\CC^\rev}\otimes(\one\boxtimes a))\odot b^L
\simeq [a^L,c^L]_{\CC\boxtimes\CC^\rev}\odot b^L$,
where we used Lemma \ref{lem:cc-inhom} in the last step.
\end{proof}

\subsection{Structure of multi-tensor categories}


\begin{thm} \label{thm:structure}
Let $\CC$ be a multi-tensor category. We have the following assertions:
\begin{enumerate}
  \item[$(1)$] The unit object $\one$ of $\CC$ is semisimple. Assume $\one=\bigoplus_{i\in\Lambda} e_i$ with $e_i$ simple. Then $e_i\otimes e_i\simeq e_i\simeq e_i^R$ and $e_i\otimes e_j\simeq0$ for $i\neq j$.
  \item[$(2)$] Let $\CC_{i j}=e_i\otimes\CC\otimes e_j$. Then $\CC \simeq \bigoplus_{i,j\in\Lambda} \CC_{i j}$ and $\CC_{i j}\otimes\CC_{j l}\subset\CC_{i l}$.
\end{enumerate}
If, in addition, $\CC$ is indecomposable, then we have
\begin{enumerate}
  \item[$(3)$] $\CC_{i j}\not\simeq0$ for all $i,j\in \Lambda$.
  \item[$(4)$] The tensor product of $\CC$ gives a $\CC_{ii}$-$\CC_{ll}$-bimodule equivalence $\CC_{i j}\boxtimes_{\CC_{j j}}\CC_{j l} \simeq \CC_{i l}$.
  \item[$(5)$] $\CC_{i j}$ is an invertible $\CC_{i i}$-$\CC_{j j}$-bimodule.
  \item[$(6)$] For each $i\in \Lambda$, $\bigoplus_{j\in\Lambda} \CC_{i j}$ is an invertible $\CC_{ii}$-$\CC$-bimodule.
  \item[$(7)$] For each $i\in \Lambda$, we have a canonical braided monoidal equivalence $\FZ(\CC)\simeq\FZ(\CC_{i i})$.
\end{enumerate}
\end{thm}

\begin{proof}
$(1)$ Let $e$ be a simple subobject of $\one$. Then $e^R$ is a simple quotient of $\one$. Since the counit map $e\otimes e^R\to\one$ is nonzero, $e\otimes e^R$ is nonzero. So, the induced injective map $e\otimes e^R \hookrightarrow \one\otimes e^R \simeq e^R$ must be an isomorphism. Similarly, $e^R\otimes e\simeq e^R$. By considering the unit map $\one\to e^R\otimes e$, we obtain $e\simeq e^R\otimes e\simeq e\otimes e^R$. Consequently, $e\otimes e\simeq e\simeq e^R$. Moreover, the composed map $e\hookrightarrow\one\twoheadrightarrow e^R$ is nonzero, because tensoring with $e$ gives an isomorphism. This shows that $e$ is a direct summand of $\one$, thus $\one$ is semisimple. Note that tensoring with $e$ annihilates all simple summands of $\one$ other than $e$. We obtain $e_i\otimes e_j\simeq0$ for $i\neq j$.

$(2)$ is a consequence of $(1)$.

$(3)$ We introduce a binary relation on $\Lambda$ defined by $i\sim j$ if $\CC_{i j}\not\simeq0$. We first show that $\sim$ is an equivalence relation. Clearly, $\sim$ is reflexive. Moreover, $\sim$ is symmetric, because $\delta^R$ induces an equivalence $\CC_{i j}\simeq\CC_{j i}$. Suppose $i\sim j$ and $j\sim l$. Choose nonzero $a\in\CC_{i j}$ and $b\in\CC_{j l}$. Then $\Hom_\CC(a\otimes b,a\otimes b) \simeq \Hom_\CC(a^L\otimes a,b\otimes b^L)$, which contains the nonzero composed morphism $a^L\otimes a\to e_j\to b\otimes b^L$. Thus $a\otimes b\not\simeq0$, i.e. $i\sim l$. This shows that $\sim$ is transitive.

Since $\CC$ is indecomposable, $\Lambda$ can only have a single equivalence class. Therefore, $\CC_{ij}\not\simeq 0$ for all $i,j\in \Lambda$.

$(4)$ Assume $\CC_{j i}=\RMod_A(\bk)$ and $\CC_{j l}=\RMod_B(\bk)$. We define $M=[A,A]_{\CC_{j j}}\simeq A\otimes A^L$ and $N=[B,B]_{\CC_{j j}}\simeq B\otimes B^L$. Then $\CC_{i j}\simeq\LMod_M(\CC_{j j})$, $\CC_{j l}\simeq\RMod_N(\CC_{j j})$, and $\CC_{i j}\boxtimes_{\CC_{j j}}\CC_{j l} \simeq \BMod_{M|N}(\CC_{j j})$. The forgetful functor $G: \BMod_{M|N}(\CC_{j j}) \to \CC_{j j}$ is right adjoint to $F:x\mapsto M\otimes x\otimes N$. The functor $G': \CC_{i l} \to \CC_{j j}$, $y\mapsto A\otimes y\otimes B^L$ is right adjoint to $F': x\mapsto A^L\otimes x\otimes B$. Clearly, both $G,G'$ are exact and conservative. Invoking the Barr-Beck theorem, we see that the functors $G$ and $G'$ exhibit $\BMod_{M|N}(\CC_{j j})$ and $\CC_{i l}$ monadic over $\CC_{j j}$, respectively. From the evident isomorphism of monads $G\circ F\simeq G'\circ F'$, we conclude that $\BMod_{M|N}(\CC_{j j}) \simeq \CC_{il} $.

Moreover, the composed equivalence $\CC_{i j}\boxtimes_{\CC_{j j}}\CC_{j l} \simeq \BMod_{M|N}(\CC_{j j})\simeq \CC_{il}$ carries $x\boxtimes_{\CC_{jj}} y$ to $x\otimes y$, which is clearly a $\CC_{ii}$-$\CC_{ll}$-bimodule functor.

$(5)$ follows from $(4)$.

$(6)$ Using a similar argument as the proof of $(4)$, we deduce the equivalences $(\bigoplus_{j\in\Lambda} \CC_{i j}) \boxtimes_\CC (\bigoplus_{k\in\Lambda} \CC_{k i}) \simeq \CC_{i i}$ and $(\bigoplus_{k\in\Lambda} \CC_{k i}) \boxtimes_{\CC_{i i}} (\bigoplus_{j\in\Lambda} \CC_{i j}) \simeq \CC$.

$(7)$ follows from $(6)$ and Proposition \ref{prop:inv-bimod-center}.
\end{proof}

\begin{rem}
The parts (1)(2) of Theorem \ref{thm:structure} were proved in \cite{egno} under a weaker condition.
\end{rem}

\void{
\begin{prop} \label{prop:surjective}
Let $\CC$ be a tensor category such that $\Hom_\CC(\one,\one)\simeq k$. The forgetful functor $F:\FZ(\CC)\to\CC$ is surjective in the sense that every object $a\in\CC$ is a subobject of some $F(b)$.
\end{prop}

\begin{proof}
Since $\Hom_\CC(\one,\one)\simeq k$, the unit $u:\one\boxtimes\one\to [\one,\one]$ is injective. So, $a$ is a subobject of $[\one,\one]\odot a$ for $a\in\CC$ by Lemma \ref{lem:cc-exact}.
\end{proof}
}

\begin{cor}  \label{cor:matrix}
Let $\CC$ be a multi-tensor category over an algebraically closed field $k$. We have $\FZ(\CC)\simeq\bk$ if and only if $\CC\simeq\Fun_\bk(\bk^n,\bk^n)$ for some integer $n\ge1$.
\end{cor}
\begin{proof}
We only need to prove the necessity. Note that $\CC$ is indecomposable by Proposition \ref{prop:center-prod}(1). Suppose $\Hom_\CC(\one,\one)\simeq k^n$.
If $\CC$ is a tensor category, then the forgetful functor $F:\FZ(\CC)\to\CC$ is surjective in the sense that every object $a\in\CC$ is a subquotient of some $F(b)$ by Proposition 3.39 in \cite{eo}. This proves the $n=1$ case. For general $n$, by Theorem \ref{thm:structure} and the $n=1$ case, we obtain $\CC_{ij}\simeq \bk$ for all $1\le i,j \le n$. Then observe that the monoidal functor $\CC\to\Fun_\bk(\bigoplus_{i=1}^n\CC_{i 1},\bigoplus_{i=1}^n\CC_{i 1})$ given by the tensor product is an equivalence.
\end{proof}

\subsection{Monoidal modules over a braided monoidal category} \label{sec:relative-over-BTC}

The following notions generalize those of a tensor category over a symmetric tensor category introduced in \cite[Definition\,4.16]{dgno}.
\begin{defn} \label{def:monoidal-modules}
Let $\CC$ and $\CD$ be finite braided monoidal categories.
\bnu
\item A {\it monoidal left $\CC$-module} is a finite monoidal category $\CM$ equipped with a $\bk$-linear braided monoidal functor $\phi_\CM:\overline\CC\to\FZ(\CM)$.
\item A {\it monoidal right $\CD$-module} is a finite monoidal category $\CM$ equipped with a $\bk$-linear braided monoidal functor $\phi_\CM:\CD\to\FZ(\CM)$.
\item A {\it monoidal $\CC$-$\CD$-bimodule} is a finite monoidal category $\CM$ equipped with a $\bk$-linear braided monoidal functor $\phi_\CM:\overline\CC\boxtimes \CD\to\FZ(\CM)$.
\enu
A monoidal $\CC$-$\CD$-bimodule is said to be {\it closed} if $\phi_\CM$ is an equivalence.
\end{defn}

\begin{rem} \label{rem:monoidal-left-right-bi}
A monoidal left $\CC$-module is precisely a monoidal right $\overline{\CC}$-module. A monoidal $\CC$-$\CD$-bimodule is precisely a monoidal right $\overline{\CC}\boxtimes \CD$-module. If $\CM$ is a monoidal $\CC$-$\CD$-bimodule, then $\CM^\rev$ is automatically a monoidal $\CD$-$\CC$-bimodule.
\end{rem}

\begin{exam}
If $\CC,\CD$ are multi-tensor categories and $\CM$ is a finite $\CC$-$\CD$-bimodule, then the category $\Fun_{\CC|\CD}(\CM,\CM)$ is naturally a monoidal $\FZ(\CC)$-$\FZ(\CD)$-bimodule (not a monoidal $\FZ(\CD)$-$\FZ(\CC)$-bimodule).
\end{exam}

\begin{rem}
In the language of $E_n$-algebras (see for example \cite{lurie2}), a monoidal category is an $E_1$-algebra in the symmetric monoidal $\infty$-category of categories, while a braided monoidal category is an $E_2$-algebra, and a symmetric monoidal category is an $E_3$-algebra. The monoidal module defined here is an $E_1$-algebra over an $E_2$-algebra. Moreover, the Drinfeld center of a monoidal category is the center of an $E_1$-algebra, and
the M\"uger center of a braided monoidal category is the center of an $E_2$-algebra.
\end{rem}

\begin{rem}
A monoidal right $\CD$-module $\CM$ is automatically equipped with a unital monoidal action $\CM \times \CD \to \CM$ induced by $\CM \times \FZ(\CM) \to \CM$. Therefore, $\CM$ is an ordinary right $\CD$-module, when we forget the braidings on $\CD$ and the monoidal structure on $\CM$.

It is also useful to encode the data $\phi_\CM: \CD \to \FZ(\CM)$ by the induced {\it central} functor $f_\CM: \CD\to \CM$ \cite{roman,dmno}. That is, $f_\CM$ is equipped with a half-braiding $f_\CM(a) \otimes y \xrightarrow{c_{a,y}} y\otimes f_\CM(a)$ for $a\in \CD,y\in \CM$ such that the following two diagrams:
$$
\xymatrix{
f_\CM(a) \otimes f_\CM(b) \otimes y \ar[rr]^-{\Id_{f_\CM(a)} \otimes c_{b,y}} & & f_\CM(a) \otimes y\otimes f_\CM(b) \ar[d]^{c_{a,y}\otimes \Id_{f_\CM(b)}} \\
f_\CM(a\otimes b) \otimes y \ar[r]^-{c_{a\otimes b, y}} \ar[u]^\sim & y \otimes f_\CM(a\otimes b) & y \otimes f_\CM(a) \otimes f_\CM(b) \ar[l]_\sim
}
$$
$$
\xymatrix{
f_\CM(a) \otimes f_\CM(b) \ar[r]^-\sim \ar[d]_{c_{a,f_\CM(b)}} & f_\CM(a\otimes b) \ar[d]^{f_\CM(\beta_{a,b})} \\
f_\CM(b) \otimes f_\CM(a) \ar[r]^-\sim & f_\CM(b\otimes a)\,
}
$$
are commutative.
\end{rem}

The following definition generalizes Definition\,2.7 in \cite{dno}.
\begin{defn} \label{def:monoidal-module-map}
Let $\CD$ be a finite braided monoidal category and $\CM,\CN$ monoidal right $\CD$-modules. A {\it monoidal $\CD$-module functor $F:\CM\to\CN$} is a $\bk$-linear monoidal functor equipped with an isomorphism of monoidal functors $F\circ f_\CM\simeq f_\CN: \CD\to\CN$ such that the evident diagram
\be  \label{diag:m-m-map}
\xymatrix{
  F(f_\CM(a)\otimes x) \ar[r]^\sim \ar[d]_\sim & F(x\otimes f_\CM(a)) \ar[d]^\sim \\
  f_\CN(a)\otimes F(x) \ar[r]^\sim & F(x)\otimes f_\CN(a) \\
}
\ee
is commutative for $a\in\CD$ and $x\in\CM$. Two right monoidal $\CD$-modules are said to be {\it equivalent} if there is an invertible monoidal $\CD$-module functor $F:\CM\to\CN$. The notions of a left and a bimodule functor are automatically defined as special cases of right module functors (recall Remark\,\ref{rem:monoidal-left-right-bi}).
\end{defn}


\begin{exam} \label{exam:invertible=braided-auto}
We give an example of monoidal bimodule equivalence. Let $\CC,\CD$ be multi-tensor categories and $\CM$ an invertible $\CC$-$\CD$-bimodule. By Proposition\,\ref{prop:inv-bimod-center}, the canonical monoidal functors $\FZ(\CC) \xrightarrow{L} \Fun_{\CC|\CD}(\CM,\CM) \xleftarrow{R} \FZ(\CD)$ are equivalences and $R^{-1}\circ L: \FZ(\CC) \to \FZ(\CD)$ preserves braidings. In this case,
\bnu
\item the monoidal functors $\FZ(\CC) \xrightarrow{L} \Fun_{\CC|\CD}(\CM,\CM) \xleftarrow{R} \FZ(\CD)$ define a monoidal $\FZ(\CC)$-$\FZ(\CD)$-bimodule structure on $\Fun_{\CC|\CD}(\CM,\CM)$;

\item the monoidal functors $\FZ(\CC) \xrightarrow{R^{-1}\circ L} \FZ(\CD) \xleftarrow{\Id} \FZ(\CD)$ define a monoidal $\FZ(\CC)$-$\FZ(\CD)$-bimodule structure on $\FZ(\CD)$.
\enu
Then the monoidal equivalence $R: \FZ(\CD) \to \Fun_{\CC|\CD}(\CM,\CM)$ defines a monoidal $\FZ(\CC)$-$\FZ(\CD)$-bimodule equivalence.
\end{exam}

\subsection{Criterion of rigidity}  \label{sec:rigidity}

Let $\CC$ be a braided multi-tensor category, $\CM$ and $\CN$ monoidal right and left $\CC$-modules, respectively. It is standard to show that $\CM\boxtimes_\CC\CN$ has a canonical structure of a monoidal category with the tensor unit given by $\one_\CM\boxtimes_\CC\one_\CN$ such that the canonical functor $\CM \boxtimes\CN \to \CM \boxtimes_\CC\CN$ is monoidal (see \cite{green}).

\begin{defn}[\cite{BD}]
An {\it r-category} is a monoidal category $\CC$ with a tensor unit $\one$ such that $\CC$ is enriched in itself and that the functor $[-,\one]:\CC\to\CC^\op$ is an equivalence. In this case, we will denote the functor $[-,\one]$ by $\FD$.
\end{defn}

\begin{rem}
Let $\CC$ be an r-category. Given an object $a\in\CC$, we have canonical morphisms $u:\one\to[a,a]$ and $v_a:\FD a\otimes a = [a,\one]\otimes a \to \one$. A morphism $a\otimes b\to\one$ induces a morphism $a\to[b,\one]=\FD b$.
\end{rem}

The following criterion of rigidity was proved in \cite{BD}.
\begin{prop}[\cite{BD}] \label{prop:rigid}
Let $\CC$ be an r-category. Then $\CC$ is rigid if and only if the composed morphism $(\FD b\otimes\FD a)\otimes(a\otimes b) \xrightarrow{\Id_{\FD b}\otimes v_a\otimes\Id_b} \FD b\otimes\one\otimes b \simeq \FD b\otimes b \xrightarrow{v_b} \one$ induces an isomorphism $\phi_{a,b}: \FD b\otimes\FD a \to \FD(a\otimes b)$ for $a,b\in\CC$.
\end{prop}
\void{
\begin{proof}
Necessity. We have $\FD a\simeq a^L$. The induced morphism $\phi_{a,b}$ coincides with the canonical isomorphism $b^L\otimes a^L \simeq (a\otimes b)^L$.

Sufficiency. We need to show that $\FD a$ is left dual to $a$. Assume $a=\FD b$. We have a morphism $u_a: \one \xrightarrow{u} [a,a] \simeq \FD(a\otimes b) \xrightarrow{\phi_{a,b}^{-\one}} a\otimes\FD a$.
The composed morphism
$$a \simeq \one\otimes a \xrightarrow{u_a\otimes\Id_a} a\otimes\FD a\otimes a \xrightarrow{\Id_a\otimes v_a} a\otimes\one \simeq a$$
coincides with the composed morphism
$$a \simeq \one\otimes a \xrightarrow{u\otimes\Id_a} [a,a]\otimes a \to a$$
which is the identity. The composed morphism
$$\FD a \simeq \FD a\otimes\one \xrightarrow{\Id_{\FD a}\otimes u_a} \FD a\otimes a\otimes\FD a \xrightarrow{v_a\otimes\Id_{\FD a}} \one\otimes\FD a \simeq \FD a$$
coincides with the composed morphism
$$\FD a \simeq \FD a\otimes\one \xrightarrow{\Id_{\FD a}\otimes u} [a,\one]\otimes[a,a] \to [a,\one] = \FD a$$
which is also the identity. This proves that $\CC$ is rigid.
\end{proof}
}

\begin{prop} \label{prop:rigidity}
Let $\CC$ be a braided multi-tensor category, $\CM$ and $\CN$ monoidal right and left $\CC$-modules, respectively. Suppose that $\CM,\CN$ are rigid. Then the finite monoidal category $\CM\boxtimes_\CC\CN$ is an r-category. Moreover, the following conditions are equivalent:
\begin{enumerate}
  \item The monoidal category $\CM\boxtimes_\CC\CN$ is rigid.
  \item The tensor product of $\CM\boxtimes_\CC\CN$ is left exact separately in each variable.
  \item The following equality holds for $a,b\in\CM\boxtimes_\CC\CN$
\begin{equation}\label{eqn:mt-rigid}
\mathrm{length}(a\otimes b) = \mathrm{length}(s(a)\otimes s(b))
\end{equation}
where $s(a)$ is the direct sum of the composition factors of $a$ (i.e. the simple objects associated to a composition series of $a$).
  \item Equation \eqref{eqn:mt-rigid} holds for $a=a'\boxtimes_\CC a''$, $b=b'\boxtimes_\CC b''$ where $a',b'\in\CM$ and $a'',b''\in\CN$ are simple.
\end{enumerate}
\end{prop}

\begin{proof}
We have a balanced $\CC$-module functor
$$\CM\times\CN \to \CN^{\op|L}\boxtimes_\CC\CM^{L|\op}, \quad (x,y)\mapsto y^L\boxtimes_\CC x^L$$
where the isomorphism $y^L\boxtimes_\CC(x\odot a)^L \simeq (a\odot y)^L\boxtimes_\CC x^L$ is the composition
$$y^L\boxtimes_\CC(x\odot a)^L
= y^L\boxtimes_\CC(x\otimes\phi_\CM(a))^L
\simeq y^L\boxtimes_\CC(\phi_\CM(a)\otimes x)^L
= y^L\boxtimes_\CC(a\odot^L x^L)
$$
$$
\simeq (y^L\odot^L a)\boxtimes_\CC x^L
= (y\otimes\phi_\CN(a))^L\boxtimes_\CC x^L
\simeq (\phi_\CN(a)\otimes y)^L\boxtimes_\CC x^L
= (a\odot y)^L\boxtimes_\CC x^L.$$
It induces an equivalence $\CM\boxtimes_\CC\CN \simeq \CN^{\op|L}\boxtimes_\CC\CM^{L|\op}$, $x\boxtimes_\CC y\mapsto y^L\boxtimes_\CC x^L$. Therefore, the formula $x\boxtimes_\CC y\mapsto x^L\boxtimes_\CC y^L$ defines an equivalence $\CM\boxtimes_\CC\CN \simeq (\CM\boxtimes_\CC\CN)^\op$ by Corollary \ref{cor:tensor-prod}(3).

Moreover, we have $\FD(x\boxtimes_\CC y) = [x\boxtimes_\CC y, \one_\CM\boxtimes_\CC\one_\CN]_{\CM\boxtimes_\CC\CN} \simeq x^L \boxtimes_\CC y^L$. Indeed, set $\CM=\RMod_M(\CC)^{\op|L}$ for some algebra $M$ in $\CC$. We have
\begin{align*}
&\Hom_{\CM\boxtimes_\CC\CN}(m\boxtimes_\CC n, [x\boxtimes_\CC y,\one_\CM\boxtimes_\CC\one_\CN]_{\CM\boxtimes_\CC\CN})  \\
&\hspace{2cm}\simeq \Hom_{\CM\boxtimes_\CC\CN}((m\otimes x)\boxtimes_\CC (n\otimes y), \one_\CM\boxtimes_\CC\one_\CN) \\
&\hspace{2cm}\simeq \Hom_\CC(\one_\CC, [n\otimes y, \one_\CN]_\CC \otimes [\one_\CM, m\otimes x]_\CC) \\
&\hspace{2cm}\simeq \Hom_\CC(\one_\CC, [n,y^L]_\CC \otimes [x^L,m]_\CC)  \\
&\hspace{2cm}\simeq \Hom_{\CM\boxtimes_\CC\CN}(m\boxtimes_\CC n, x^L\boxtimes_\CC y^L),
\end{align*}
where the second and fourth steps are due to Corollary \ref{cor:tensor-prod}(2) and the third step follows from the following identities, for $c\in \CC$,
\begin{align*}
\Hom_\CC(c, [\one_\CM, m\otimes x]_\CC) &\simeq \Hom_{\RMod_M(\CC)}(c\otimes \one_\CM, m\otimes x) \simeq \Hom_\CM(m\otimes x, c\otimes \one_\CM) \\
&\simeq \Hom_\CM(m, c\otimes x^L)  \simeq \Hom_{\RMod_M(\CC)}(c\otimes x^L, m) \\
&\simeq \Hom_\CC(c,[x^L,m]_\CC).
\end{align*}
Therefore, $\CM\boxtimes_\CC\CN$ is an r-category.

$(1)\Rightarrow(2)\Leftrightarrow(3)\Rightarrow(4)$ are clear.

$(2)\Rightarrow(1)$. Since $\CM,\CN$ are rigid, the induced morphism $\phi_{a,b}: \FD b\otimes\FD a \to \FD(a\otimes b)$ is an isomorphism if $a,b$ lie in the essential image of the canonical functor $\CM\times\CN \to \CM\boxtimes_\CC\CN$. Since the tensor product of $\CM\boxtimes_\CC\CN$ is exact separately in each variable, both $\FD b\otimes\FD a$ and $\FD(a\otimes b)$ are exact in each variable $a,b$. So, $\phi_{a,b}$ is an isomorphism for general $a,b$. By Proposition \ref{prop:rigid}, $\CM\boxtimes_\CC\CN$ is rigid.

$(4)\Rightarrow(3)$. Since the tensor product of $\CM\boxtimes_\CC\CN$ is right exact separately in each variable, the direction $\le$ of \eqref{eqn:mt-rigid} is always true. Hence when \eqref{eqn:mt-rigid} holds, it still holds if we replace $a,b$ by their subquotients. That is, it suffices to consider the case where $a,b$ lie in the essential image of the canonical functor $\CM\times\CN \to \CM\boxtimes_\CC\CN$, say, $a=a'\boxtimes_\CC a''$, $b=b'\boxtimes_\CC b''$. We obtain the following identities:
\begin{align*}
\mathrm{length}(a\otimes b)&=\mathrm{length}\left( (s(a'\otimes b') \boxtimes_\CC s(a''\otimes b'') \right) \\
&=\mathrm{length}\left( (s(a')\otimes s(b')) \boxtimes_\CC (s(a'')\otimes s(b'')) \right) \\
&=\mathrm{length}\left( (s(a')\boxtimes_\CC s(a''))\otimes (s(b')\boxtimes_\CC s(b'')) \right) \\
&=\mathrm{length}\left( s(a) \otimes s(b) \right),
\end{align*}
where the first identity is due to the exactness of $\boxtimes_\CC$ in each variable (recall Remark\,\ref{rem:tensor-product-exact}), the second identity is due to the exactness of $\otimes$ in $\CM$ and in $\CN$, and the last identity is due to Condition $(4)$.
\end{proof}

\begin{cor} \label{cor:rigidity}
Let $\CC$ be a braided multi-tensor category, $\CM$ and $\CN$ monoidal right and left $\CC$-modules, respectively. Suppose that $\CM,\CN$ are rigid. Then the finite monoidal category $\CM\boxtimes_\CC\CN$ is rigid if $x\boxtimes_\CC y$ is semisimple for simple objects $x\in\CM$ and $y\in\CN$.
\end{cor}

\begin{exam}\label{exam:rigid-boxtensor}
If $\CC,\CD$ are multi-tensor categories over an algebraically closed field $k$, then $\CC\boxtimes\CD$ is also a multi-tensor category.
Actually, since $k$ is algebraically closed, $x\boxtimes y$ is simple for simple objects $x\in\CC$, $y\in\CD$. So, $\CC\boxtimes\CD$ is rigid by Corollary \ref{cor:rigidity}.
\end{exam}

\begin{rem}
Even when $\CC,\CD$ are multi-tensor categories, the finite monoidal category $\CC\boxtimes\CD$ may be not rigid. For example, let $k'$ be a nonseparable extension of $k$. The finite monoidal category $\bk'\boxtimes\bk'$ is not rigid due to Theorem \ref{thm:structure} and the fact that the unit object of $\bk'\boxtimes\bk'$ is not semisimple. 
\end{rem}

\section{The center functor} \label{sec:center-functor}

In this section, we assume that $k$ is an algebraically closed field.

\subsection{Functoriality of center} \label{sec:fun-center}

Given a multi-tensor category $\CC$ with a tensor unit $\one$, we use $I_\CC$ to denote the right dual of $[\one,\one]_{\CC^\rev\boxtimes\CC}\in\CC^\rev\boxtimes\CC$.

Given a left module $\CM$ over a rigid monoidal category $\CC$, we use ${}^{LL}_{~\CC}\CM$ to denote the left $\CC$-module which has the same underlying category as $\CM$ but equipped with the action $a\odot^{LL} x = a^{LL}\odot x$. Given a right module $\CN$ over a rigid monoidal category $\CC$, we use $\CN^{RR}_\CC$  to denote ${}^{~LL}_{\CC^\rev}\CN$.

\begin{lem}  \label{lem:fun:lem1}
Let $\CC$ be a multi-tensor category and $\CM,\CM'$ finite left $\CC$-modules. There are an equivalence
\begin{equation}\label{eq:CMLLM1}
\CC^{\op|L} \boxtimes_{\CC^\rev\boxtimes\CC} \Fun_\bk(\CM,\CM') \simeq \Fun_\CC(\CM,\CM'),
\end{equation}
which carries $\one_\CC \boxtimes_{\CC^\rev\boxtimes\CC} f$ to $I_\CC \odot f$, and an equivalence
\be  \label{eq:CMLLM2}
\CC \boxtimes_{\CC^\rev\boxtimes\CC} \Fun_\bk({}^{LL}_{~\CC}\CM,\CM') \simeq \Fun_\CC(\CM,\CM'),
\ee
which carries $\one_\CC \boxtimes_{\CC^\rev\boxtimes\CC} f$ to $I_\CC \odot f$.
\end{lem}

\begin{proof}
By Corollary \ref{cor:tensor-prod}(1) and the definition of $I_\CC$, we see that the composed equivalence
$$
\CC^{\op|L} \boxtimes_{\CC^\rev\boxtimes\CC} \Fun_\bk(\CM,\CM')
\simeq \Fun_{\CC^\rev\boxtimes\CC}(\CC,\Fun_\bk(\CM,\CM'))
\simeq \Fun_\CC(\CM,\CM')
$$
carries $\one_\CC \boxtimes_{\CC^\rev\boxtimes\CC} f \mapsto [-,\one_\CC]^R \odot f \mapsto I_\CC \odot f$. Thus we obtain the first equivalence (\ref{eq:CMLLM1}).
Then using the equivalence $\CC\simeq\CC^\op$, $a\mapsto a^L$, we obtain (\ref{eq:CMLLM2}) from (\ref{eq:CMLLM1}). Note that (\ref{eq:CMLLM2}) carries $\one_\CC \boxtimes_{\CC^\rev\boxtimes\CC}$ to $I_\CC \odot f$.
\end{proof}


\begin{rem} \label{rem:fun:rem1}
If $\CN, \CN'$ are finite right $\CC$-modules, we have
\be
\CC^\rev \boxtimes_{\CC\boxtimes \CC^\rev} \Fun_\bk({}^{~LL}_{\CC^{\rev}}\CN, \CN') \simeq \Fun_{\CC^\rev}(\CN, \CN'),
\ee
which maps $\one_{\CC^\rev} \boxtimes_{\CC\boxtimes \CC^\rev} g \mapsto I_{\CC^\rev} \odot g$, or equivalently,
\be \label{eq:NrrNC}
\Fun_\bk(\CN_\CC^{RR}, \CN') \boxtimes_{\CC^\rev\boxtimes \CC} \CC \simeq \Fun_{\CC^\rev}(\CN, \CN')
\ee
which maps $g\boxtimes_{\CC^\rev\boxtimes \CC} \one_\CC \mapsto I_{\CC^\rev} \odot g$.
\end{rem}

\begin{lem}  \label{lem:fun:lem2}
Let $\CC,\CD$ be multi-tensor categories, $\CM$ a finite left $\CC$-module, $\CN$ a finite $\CC$-$\CD$-bimodule and $\CP$ a finite left $\CD$-module. There is an equivalence
\begin{equation*}  
\Fun_{\CC}(\CM, \CN)\boxtimes_\CD\CP \simeq \Fun_{\CC}(\CM, \CN\boxtimes_\CD\CP), \quad f\boxtimes_\CD y \mapsto f(-)\boxtimes_\CD y.
\end{equation*}
\end{lem}

\begin{proof}
Using Corollary \ref{cor:tensor-prod}(1) twice, we obtain a composed equivalence
\begin{align} \label{eq:M-N6}
\Fun_{\CC}(\CM, \CN)\boxtimes_\CD\CP
&\simeq \CM^{\op|L} \boxtimes_{\CC} \CN \boxtimes_{\CD} \CP
\simeq \Fun_{\CC}(\CM, \CN\boxtimes_\CD\CP)   \\
([-,x]^R\odot x')\boxtimes_\CD y &\mapsto x\boxtimes_\CC x' \boxtimes_\CD y \mapsto
[-,x]^R \odot (x'\boxtimes_\CD y). \nonumber
\end{align}
By the universal property of $\boxtimes_\CC$ and $\boxtimes_\CD$, $(\ref{eq:M-N6})$ carries $f\boxtimes_\CD y$ to $f(-)\boxtimes_\CD y$.
\end{proof}

\begin{lem}  \label{lem:fun:lem3}
Let $\CA, \CB, \CC,\CD$ be multi-tensor categories and ${}_{\CA}\CM_\CC$, ${}_\CA\CM'_\CD$, ${}_\CC\CN_\CB$, ${}_\CD\CN'_\CB$ finite bimodules. There is an equivalence
\begin{equation}  \label{eq:AMMCDNNB}
\Fun_\CA(\CM, \CM') \boxtimes_{\CC^\rev\boxtimes\CD} \Fun_{\CB^\rev}({}^{LL}_{~\CC}\CN, \CN') \simeq \Fun_{\CA|\CB}(\CM\boxtimes_\CC \CN, \CM'\boxtimes_\CD \CN')
\end{equation}
defined by $f\boxtimes_{\CC^\rev\boxtimes\CC}g \mapsto I_\CC\odot (f(-)\boxtimes_\CD g(-))$.
\end{lem}

\begin{proof}
We first prove the Lemma for the special case $\CA=\bk=\CB$. We have the following composed equivalence:
\begin{align}
\Fun_\bk(\CM,\CM') &\boxtimes_{\CC^\rev\boxtimes\CD} \Fun_\bk({}^{LL}_{~\CC}\CN,\CN') \nn
&\simeq \CC \boxtimes_{\CC\boxtimes\CC^\rev} (\Fun_\bk(\CM,\CM') \boxtimes_\CD \Fun_\bk({}^{LL}_{~\CC}\CN,\CN')) \nn
&\simeq \CC \boxtimes_{\CC^\rev\boxtimes\CC} \Fun_\bk({}^{LL}_{~\CC}\CN,\,\Fun_\bk(\CM,\CM')\boxtimes_\CD \CN') \nn
&\simeq \Fun_\CC(\CN,\Fun_\bk(\CM,\CM')\boxtimes_\CD \CN') \nn
&\simeq \Fun_\CC(\CN,\Fun_\bk(\CM,\CM'\boxtimes_\CD\CN')) \nn
&\simeq \Fun_\bk(\CM\boxtimes_\CC\CN, \CM'\boxtimes_\CD\CN'), \nonumber
\end{align}
which carries objects as follows:
\begin{align*}
 &f\boxtimes_{\CC^\rev\boxtimes\CD}g \mapsto \one_\CC \boxtimes_{\CC\boxtimes\CC^\rev} (f\boxtimes_\CD g)  
\mapsto \one_\CC \boxtimes_{\CC^\rev\boxtimes\CC} (f \boxtimes_\CD g(-)) \nn
& \mapsto I_\CC \odot (f\boxtimes_\CD g(-)) \mapsto I_\CC \odot (f(-)\boxtimes_\CD g(-)) \mapsto I_\CC \odot (f(-)\boxtimes_\CD g(-)).
\end{align*}
Here, we used Lemma \ref{lem:fun:lem2} in the second and fourth steps, and used Lemma \ref{lem:fun:lem1} in the third step.

By using the canonical equivalences $\Fun_\CA(\CM, \CM')\simeq \CA\boxtimes_{\CA^\rev\boxtimes \CA}\Fun_\bk({}^{LL}_{~\CA}\CM, \CM')$ and $\Fun_{\CB^\rev}(\CN, \CN') \simeq\Fun_\bk(\CN^{RR}_\CB, \CN')\boxtimes_{\CB^\rev\boxtimes \CB} \CB$ given by Lemma \ref{lem:fun:lem1} and Remark \ref{rem:fun:rem1}, we can reduce the general case to the above special case, thus complete the proof.
\end{proof}

\begin{rem}
The natural equivalence in (\ref{eq:AMMCDNNB}) induces a natural equivalence:
\be \label{eq:AMMCDNNB-2}
\Fun_\CA(\CM^{RR}_\CC, \CM') \boxtimes_{\CC^\rev\boxtimes\CD} \Fun_{\CB^\rev}(\CN, \CN') \simeq \Fun_{\CA|\CB}(\CM\boxtimes_\CC \CN, \CM'\boxtimes_\CD \CN')
\ee
defined by $f\boxtimes_{\CC^\rev\boxtimes\CD} g\mapsto I_{\CC^\rev}\odot(f(-)\boxtimes_\CD g(-))$. Actually, it can be obtained from (\ref{eq:AMMCDNNB}) by replacing $\CA,\CB,\CC,\CD$ by their $(-)^\rev$, respectively, exchanging the letter $\CA$ with $\CB$ and exchanging the letter $\CM$ with $\CN$.
\void{
Replacing $\CA,\CB,\CC,\CD$ in (\ref{eq:AMMCDNNB}) by their $(-)^\rev$, respectively, we obtain a natural equivalence:
\be \label{eq:AMMCDNNB-2}
\Fun_\CB(\CN^{RR}, \CN') \boxtimes_{\CC^\rev\boxtimes\CD} \Fun_{\CA^\rev}(\CM, \CM') \simeq \Fun_{\CB|\CA}(\CN\boxtimes_\CC \CM, \CN'\boxtimes_\CD \CM')
\ee
defined by $g\boxtimes_{\CC^\rev\boxtimes\CD}f \mapsto I_{\CC^\rev}\odot (f(-)\boxtimes_{\CD^\rev} g(-))=I_\CC\odot(g(-)\boxtimes_\CD f(-))$.
}
\end{rem}


\begin{lem}  \label{lem:fun:lem4}
Let $\CC$ be an indecomposable multi-tensor category. The formula $a\boxtimes_{\FZ(\CC)}b \mapsto a\otimes - \otimes b$ defines an equivalence $\CC \boxtimes_{\FZ(\CC)} \CC \simeq \Fun_\bk(\CC,\CC)$ as $\CC^\rev\boxtimes\CC$-$\CC^\rev\boxtimes\CC$-bimodules.
\end{lem}

\begin{proof}
Theorem 7.12.11 in \cite{egno} says that, for a faithful exact module $\CM$ over a multi-tensor category $\CD$, the canonical functor $\CD \to \Fun_{\Fun_\CD(\CM,\CM)}(\CM,\CM)$ is a monoidal equivalence. When $\CD=\CC^\rev\boxtimes\CC$ and $\CM=\CC$,  since $\FZ(\CC)=\Fun_{\CC^\rev\boxtimes\CC}(\CC,\CC)$, the canonical functor $\CC^\rev\boxtimes\CC \to \Fun_{\FZ(\CC)}(\CC,\CC)$ is a monoidal equivalence. Here, $\CC$ is an exact $\CC^\rev\boxtimes\CC$-module (see Example\,7.5.5 in \cite{egno}), and we require $\CC$ to be indecomposable to ensure that the left $\CC^\rev\boxtimes\CC$-module $\CC$ is faithful \cite[Sec.\,7.12]{egno}.
The composed equivalence (as categories)
$$
\CC \boxtimes_{\FZ(\CC)} \CC
\xrightarrow{\delta^R\boxtimes\Id} \CC^{\op|R} \boxtimes_{\FZ(\CC)} \CC
\simeq \Fun_{\FZ(\CC)}(\CC,\CC)
\simeq \CC^\rev\boxtimes\CC
$$
carries $a\boxtimes_{\FZ(\CC)}b \mapsto a^R\boxtimes_{\FZ(\CC)}b \mapsto [a^R,-]_{\FZ(\CC)}\odot b \mapsto [a,b^L]^R$, where we have used Proposition \ref{prop:tensor-prod2} in the second step and Proposition \ref{prop:center-cc} in the last step.

The following composed equivalence
\begin{align*}
\CC^\rev\boxtimes\CC
&\xrightarrow{\delta^R\boxtimes\Id} \CC^\op \boxtimes \CC
\xrightarrow{\simeq} \Fun_\bk(\CC,\CC) \\
c\boxtimes d \quad &\,\,\,\,\, \mapsto \quad  \,\, c^R\boxtimes d \,\,\, \mapsto \,\,\, \Hom_\CC(-,c^R)^R\odot d
\end{align*}
maps $[a,b^L]^R$ to a functor $f\in \Fun_\bk(\CC,\CC)$. Note that
$\Hom_\CC(\Hom_\CC(x,c^R)^R\odot d, y) \simeq \Hom_\CC(c,x^L)\otimes_k \Hom_\CC(d,y) \simeq \Hom_{\CC^\rev\boxtimes \CC}(c\boxtimes d, x^L\boxtimes y)$, which implies
$$
\Hom_\CC(f(x),y) \simeq \Hom_{\CC^\rev\boxtimes\CC}([a,b^L]^R,x^L\boxtimes y) \simeq \Hom_\CC(a\otimes x\otimes b,y),
$$
i.e. $f\simeq a\otimes-\otimes b$. Note that $a\boxtimes_{\FZ(\CC)}b \mapsto a\otimes - \otimes b$ is a $\CC^\rev\boxtimes\CC$-$\CC^\rev\boxtimes\CC$-bimodule functor.
\end{proof}

\begin{thm}  \label{thm:M-N}
Let $\CC,\CD,\CE$ be multi-tensor categories, and $\CM,\CM'$ be finite $\CC$-$\CD$-bimodules and $\CN,\CN'$ be finite $\CD$-$\CE$-bimodules. Assume $\CD$ is indecomposable. The assignment $f\boxtimes_{\FZ(\CD)}g \mapsto f\boxtimes_\CD g$ determines an equivalence of $\FZ(\CC)$-$\FZ(\CE)$-bimodules
\begin{equation} \label{eq:M-N}
\Fun_{\CC|\CD}(\CM, \CM') \boxtimes_{\FZ(\CD)} \Fun_{\CD|\CE}(\CN, \CN') \simeq \Fun_{\CC |\CE}(\CM\boxtimes_\CD \CN, \CM'\boxtimes_\CD \CN').
\end{equation}
Moreover, when $\CM=\CM'$ and $\CN=\CN'$, \eqref{eq:M-N} is an equivalence of monoidal $\FZ(\CC)$-$\FZ(\CE)$-bimodules.
\end{thm}

\begin{proof}
We have the following composed equivalence:
\begin{align*}
\Fun&_{\CC|\CD}(\CM, \CM') \boxtimes_{\FZ(\CD)} \Fun_{\CD|\CE}(\CN, \CN') \nn
& \simeq
\Fun_\CC(\CM^{RR}_\CD,\CM') \boxtimes_{\CD^\rev\boxtimes\CD} \CD \boxtimes_{\FZ(\CD)}
\CD \boxtimes_{\CD^\rev\boxtimes\CD} \Fun_{\CE^\rev}({}^{LL}_{~\CD}\CN,\CN')  \nn
& \simeq \Fun_\CC(\CM^{RR}_\CD,\CM') \boxtimes_{\CD^\rev\boxtimes\CD} \Fun_\bk(\CD,\CD) \boxtimes_{\CD^\rev\boxtimes\CD} \Fun_{\CE^\rev}({}^{LL}_{~\CD}\CN,\CN') \nn
&\simeq \Fun_\CC(\CM\boxtimes_\CD \CD,\CM'\boxtimes_\CD \CD) \boxtimes_{\CD^\rev\boxtimes\CD} \Fun_{\CE^\rev}({}^{LL}_{~\CD}\CN,\CN') \nn
&\simeq \Fun_\CC(\CM,\CM') \boxtimes_{\CD^\rev\boxtimes\CD} \Fun_{\CE^\rev}({}^{LL}_{~\CD}\CN,\CN') \nn
& \simeq \Fun_{\CC|\CE}(\CM \boxtimes_\CD \CN, \CM' \boxtimes_\CD \CN'),
\end{align*}
where we have used (\ref{eq:NrrNC}) and (\ref{eq:CMLLM2}) in the first step; Lemma \ref{lem:fun:lem4} in the second step;  (\ref{eq:AMMCDNNB-2}) in the third step; (\ref{eq:AMMCDNNB}) in the last step. For $f'\in\Fun_\CC(\CM,\CM')$, $g'\in\Fun_{\CE^\rev}(\CN,\CN')$, this composed equivalence carries
\begin{align*}
(I_{\CD^\rev} &\odot f') \boxtimes_{\FZ(\CD)} (I_\CD \odot g')  \\
&\mapsto
(f' \boxtimes_{\CD^\rev\boxtimes\CD} \one_\CD) \boxtimes_{\FZ(\CD)} (\one_\CD\boxtimes_{\CD^\rev\boxtimes \CD} g')  \\
&\mapsto
f' \boxtimes_{\CD^\rev\boxtimes\CD} \Id_\CD
\boxtimes_{\CD^\rev\boxtimes \CD} \, g' \\
&\mapsto  I_{\CD^\rev} \odot (f'(-)\boxtimes_\CD \Id_\CD(-)) \boxtimes_{\CD^\rev\boxtimes \CD} g' \\
&\mapsto  (I_{\CD^\rev} \odot f')  \boxtimes_{\CD^\rev\boxtimes \CD} g' \\
&\mapsto  I_\CD \odot ((I_{\CD^\rev} \odot f')(-) \boxtimes_\CD g'(-)) \\
& =  (I_{\CD^\rev} \odot f') \boxtimes_\CD (I_\CD \odot g')\, .
\end{align*}
Therefore, it must carry $f\boxtimes_{\FZ(\CD)}g$ to $f\boxtimes_\CD g$. It is clear that it is also a $\FZ(\CC)$-$\FZ(\CE)$-bimodule equivalence.

When $\CM=\CM'$ and $\CN=\CN'$, the formula $f\boxtimes_{\FZ(\CD)}g \mapsto f\boxtimes_\CD g$ clearly defines a monoidal equivalence. Moreover, it is routine to check that the diagram (\ref{diag:m-m-map}) is commutative in this case. Therefore, the isomorphism defined in (\ref{eq:M-N}) is a monoidal $\FZ(\CC)$-$\FZ(\CE)$-bimodule equivalence.
\end{proof}

We introduce two categories $\mtc$ and $\bmtc$ as follows:
\bnu
\item The category $\mtc$ of indecomposable multi-tensor categories over $k$ with morphisms given by the equivalence classes of finite bimodules.
\item The category $\bmtc$ of braided tensor categories over $k$ with morphisms given by the equivalence classes of monoidal bimodules.
\enu
In view of Example \ref{exam:rigid-boxtensor}, both categories are symmetric monoidal categories with respect to $\boxtimes$.

\begin{thm} \label{thm:functorial}
The assignment $\CC \mapsto \FZ(\CC)$, ${}_\CC\CM_\CD \mapsto \Fun_{\CC|\CD}(\CM,\CM)$ defines a symmetric monoidal functor $\FZ: \mtc \to \bmtc$.
\end{thm}

\begin{proof}
By Remark \ref{rem:rigid-center} and Theorem \ref{thm:M-N}, the functor $\FZ$ is well-defined. By Proposition \ref{prop:prod-fun}(2) and Proposition \ref{prop:center-prod}, $\FZ$ is a symmetric monoidal functor.
\end{proof}

\begin{rem}
The domain of $\FZ$ can not be generalized to decomposable multi-tensor categories. For example, let $\CC=\bk\oplus\bk$, regarded as a $\bk$-$\CC$-bimodule and a $\CC$-$\bk$-bimodule. Note that $\Fun_{\bk|\CC}(\CC,\CC) \boxtimes_{\FZ(\CC)} \Fun_{\CC|\bk}(\CC,\CC) \simeq \CC$ and $\Fun_{\bk|\bk}(\CC\boxtimes_\CC\CC,\CC\boxtimes_\CC\CC) \simeq \CC\oplus\CC$ do not match.
\end{rem}


\subsection{An alternative approach to a monoidal functor}  \label{sec:monoidal-functor}

In this subsection, we provide an equivalent definition of a $\bk$-linear monoidal functor. It is inspired by the notion of a morphism between two topological orders introduced in \cite{kong-wen-zheng}.

\medskip
Let $f:\CC\to\CD$ be a $\bk$-linear monoidal functor between two multi-tensor categories $\CC$ and $\CD$. It endows the category $\CD$ with a canonical $\CC$-$\CC$-bimodule structure, denoted by ${}_f\CD_f$. The category $\Fun_{\CC|\CC}(\CC, \CD)$ has a natural structure of a monoidal category \cite{gnn} which can be described as follows.
An object is a pair:
$$
(d\in\CD, \beta_{d,-}=\{d\otimes f(c) \xrightarrow{\beta_{d,c}} f(c) \otimes d\}_{c\in \CC})
$$
where $\beta_{d,-}$ is a half-braiding (i.e. a natural isomorphism in the variable $c\in \CC$ and satisfying $\beta_{d,c'} \circ \beta_{d,c} = \beta_{d,c\otimes c'}$). A morphism $(d,\beta_{d,-})\to(d',\beta_{d',-})$ is defined by a morphism $\psi: d \to d'$ preserving half-braidings. The monoidal structure is given by the formula $(d,\beta_{d,-})\otimes(d',\beta_{d',-})=(d\otimes d',\beta_{d,-}\circ\beta_{d',-})$.

It is not hard to see that an object in $\Fun_{\CC|\CD}({}_f\CD, {}_f\CD)$ is also such a pair. Therefore, we simply identify these two monoidal categories:
$$
\Fun_{\CC|\CC}(\CC, {}_f\CD_f)=\Fun_{\CC|\CD}({}_f\CD, {}_f\CD).
$$

\begin{lem}   \label{lem:M-N}
Let $f:\CC\to\CD$ be a $\bk$-linear monoidal functor from an indecomposable multi-tensor category $\CC$ to a finite monoidal category $\CD$. Then the evaluation functor $\CC\times\Fun_{\CC|\CC}(\CC,\CD) \to \CD$, $(c,j)\mapsto j(c)$ induces an equivalence of monoidal right $\FZ(\CD)$-modules
\begin{equation}   \label{eq:C-fun-CD=D}
\CC\boxtimes_{\FZ(\CC)}\Fun_{\CC|\CC}(\CC,{}_f\CD_f) \simeq \CD.
\end{equation}
\end{lem}

\begin{proof}
This follows from Theorem \ref{thm:M-N} and the obvious monoidal equivalences $\Fun_{\bk|\CC}(\CC,\CC)\simeq\CC$ and $\Fun_{\bk|\CC}(\CC,\CD_f) \simeq \CD$.
\end{proof}


Let $\CC$ and $\CD$ be finite monoidal categories. We define two groupoids.
One is the groupoid $\Fun^\otimes(\CC, \CD)$ of $\bk$-linear monoidal functors from $\CC$ to $\CD$ and isomorphisms between them. The other one $\Fun^\ph(\CC, \CD)$ is defined below.
\begin{defn}
The groupoid $\Fun^\ph(\CC, \CD)$ consists of
\begin{itemize}
\item objects: an object $f\in \Fun^\ph(\CC, \CD)$ is a pair $f=(f^{(0)}, f^{(1)})$ where $f^{(0)}$ is a monoidal $\FZ(\CC)$-$\FZ(\CD)$-bimodule, and $f^{(1)}: \CC \boxtimes_{\FZ(\CC)} f^{(0)} \xrightarrow{\simeq} \CD$ is an equivalence of monoidal right $\FZ(\CD)$-modules (recall Definition \ref{def:monoidal-module-map}).

\item isomorphisms: an isomorphism $\phi: f\to g$ in $\Fun^\ph(\CC, \CD)$ is an equivalence class of pairs $\phi=(\phi^{(0)}, \phi^{(1)})$, where $\phi^{(0)}: f^{(0)} \to g^{(0)}$ is an equivalence of monoidal $\FZ(\CC)$-$\FZ(\CD)$-bimodules and
$$
\phi^{(1)}: g^{(1)} \circ (\Id_{\CC} \boxtimes_{\FZ(\CC)} \phi^{(0)}) \xrightarrow{\simeq} f^{(1)}
$$
is a monoidal natural isomorphism such that the following diagram:
\be  \label{diag:phi-(1)}
\raisebox{1em}{\xymatrix@R=0.2em{
& \CC \boxtimes_{\FZ(\CC)}  g^{(0)} \ar[ddr]^{g^{(1)}} &  \\
& 
&  \\
\CC \boxtimes_{\FZ(\CC)} f^{(0)} \ar[uur]^{\hspace{-2cm} \Id_{\CC} \boxtimes_{\FZ(\CC)} \phi^{(0)} } \ar[rr]^{f^{(1)}}  & &   \CD
}}
\ee
is commutative up to $\phi^{(1)}$.
Two pairs $(\phi_i^{(0)},\phi_i^{(1)})$ for $i=1,2$ are isomorphic if there is an isomorphism $\psi: \phi_1^{(0)} \to \phi_2^{(0)}$ such that $\phi_1^{(1)}=\phi_2^{(1)}\circ\psi$.
\end{itemize}
\end{defn}

\begin{thm}  \label{thm:fun=ph-mor}
Let $\CC$ be an indecomposable multi-tensor category and $\CD$ a finite monoidal category. There is an equivalence between two groupoids $\Fun^\otimes(\CC, \CD)$ and $\Fun^\ph(\CC, \CD)$ defined by, for a $\bk$-linear monoidal functor $f: \CC \to \CD$,
$$
f \mapsto \tilde{f}= \Big(\, \tilde{f}^{(0)}=\Fun_{\CC|\CC}(\CC, {}_f\CD_f), \quad \CC\boxtimes_{\FZ(\CC)} \Fun_{\CC|\CC}(\CC, {}_f\CD_f) \xrightarrow{\tilde{f}^{(1)}} \CD \Big),
$$
where ${}_f\CD_f$ is the $\CC$-$\CC$-bimodule structure on $\CD$ induced from the $\bk$-linear monoidal functor $f: \CC \to \CD$ and $\tilde{f}^{(1)}$ is given by the monoidal equivalence (\ref{eq:C-fun-CD=D}).
\end{thm}

\begin{proof}
Let $\phi: f \to g$ be an isomorphism between two $\bk$-linear monoidal functors from $\CC$ to $\CD$. Then ${}_{f}\CD_f \simeq {}_{g}\CD_g$ as $\CC$-$\CC$-bimodules canonically (simply by the identity functor). It further induces a monoidal equivalence $\tilde{\phi}^{(0)}:\tilde{f}^{(0)} \simeq  \tilde{g}^{(0)}$ and a natural isomorphism
$$\tilde{\phi}^{(1)}_{c\boxtimes j}: j(c)\simeq f(c)\otimes j(\one_\CC) \xrightarrow{\phi_c\otimes\Id} g(c)\otimes j(\one_\CC) \simeq \tilde{\phi}^{(0)}(j)(c)$$
for $c\in \CC$, $j\in \Fun_{\CC|\CC}(\CC,{}_{f}\CD_f)$ such that the diagram (\ref{diag:phi-(1)}) commutes up to $\tilde{\phi}^{(1)}$. Therefore, we obtain a functor $\Fun^\otimes(\CC, \CD) \to \Fun^\ph(\CC,\CD)$ that carries $f$ to $\tilde{f}$ and $\phi$ to $\tilde{\phi}=(\tilde{\phi}^{(0)}, \tilde{\phi}^{(1)})$.

\smallskip
Conversely, let $g \in \Fun^\ph(\CC, \CD)$, i.e.
$$
g=(g^{(0)}, \quad \CC \boxtimes_{\FZ(\CC)} g^{(0)} \xrightarrow{g^{(1)}} \CD),
$$
where $g^{(0)}$ is a monoidal $\FZ(\CC)$-$\FZ(\CD)$-bimodule and $g^{(1)}$ an equivalence of monoidal right $\FZ(\CD)$-modules. Let $\one_{g^{(0)}}$ be the tensor unit of $g^{(0)}$. The functor $x \mapsto x \boxtimes_{\FZ(\CC)} \one_{g^{(0)}}$ defines a $\bk$-linear monoidal functor $\CC \to \CC \boxtimes_{\FZ(\CC)} g^{(0)}$. Then we obtain a composed $\bk$-linear monoidal functor
$$
\underline{g}: \CC \to \CC \boxtimes_{\FZ(\CC)} g^{(0)} \xrightarrow{g^{(1)}} \CD.
$$
Suppose we are given an isomorphism $\phi:f\to g$ in $\Fun^\ph(\CC, \CD)$. Then we see from diagram (\ref{diag:phi-(1)}) that $\phi^{(1)}$ induces an isomorphism of $\bk$-linear monoidal functors $\underline{\phi}:\underline{f}\to\underline{g}$, and it is clear that $\underline{\phi}$ is independent of the representative of $\phi$.
Therefore, we obtain a functor $\Fun^\ph(\CC, \CD) \to \Fun^\otimes(\CC,\CD)$ that carries $g$ to $\underline{g}$ and carries $\phi$ to $\underline{\phi}$.

\smallskip
Given a $\bk$-linear monoidal functor $f: \CC \to \CD$, the $\bk$-linear monoidal functor $\underline{(\tilde{f})}: \CC \to \CD$ is given by $c\mapsto c\boxtimes_{\FZ(\CC)}\one_{\tilde{f}^{(0)}} \mapsto f(c)\otimes f(\one_\CC) \simeq f(c)$ for $c\in \CC$.
Therefore, we obtain a natural isomorphism $\underline{(\tilde{f})}\simeq f$.

Suppose $g \in \Fun^\ph(\CC, \CD)$. We need to show 
$g\simeq\widetilde{(\underline{g})}$ to complete the proof. We have the following equivalences of monoidal $\FZ(\CC)$-$\FZ(\CD)$-bimodules:
\begin{align} \label{eq:proof-eq-two-def-morphisms}
g^{(0)} &\simeq \FZ(\CC) \boxtimes_{\FZ(\CC)} g^{(0)} \simeq \Fun_{\CC|\CC}(\CC,\CC) \boxtimes_{\FZ(\CC)} g^{(0)} \nn
&\simeq \Fun_{\CC|\CC}(\CC,\CC\boxtimes_{\FZ(\CC)} g^{(0)})
\simeq \Fun_{\CC|\CC}(\CC, {}_{\underline{g}}\CD_{\underline{g}})
= \widetilde{(\underline{g})}^{(0)},
\end{align}
where we have used Lemma \ref{lem:fun:lem2} in the third step. Let $\phi^{(0)}$ be the composed isomorphism defined by (\ref{eq:proof-eq-two-def-morphisms}). We claim that the following diagram:
\be \label{map:phi-1+prop-2}
\xymatrix{
\CC\boxtimes_{\FZ(\CC)} g^{(0)} \ar[rr]^{\hspace{-0.5cm}\Id_\CC \boxtimes_{\FZ(\CC)} \phi^{(0)}} \ar[rrd]_{g^{(1)}} & &
\CC\boxtimes_{\FZ(\CC)} \Fun_{\CC|\CC}(\CC, {}_{\underline{g}}\CD_{\underline{g}}) \ar[d]^{\widetilde{(\underline{g})}^{(1)}} \\
& &  \CD,
}
\ee
where $\widetilde{(\underline{g})}^{(1)}$ is defined as in \eqref{eq:C-fun-CD=D}, is commutative up to a canonical natural isomorphism. Indeed, consider the following commutative diagram:
\be 
\xymatrix{
\CC\boxtimes_{\FZ(\CC)} g^{(0)}  \ar[d]_\simeq \ar[r]^-{\Id_\CC \boxtimes_{\FZ(\CC)} \phi^{(0)}} &
\CC\boxtimes_{\FZ(\CC)} \Fun_{\CC|\CC}(\CC, {}_{\underline{g}}\CD_{\underline{g}}) \ar[r]^-{\widetilde{(\underline{g})}^{(1)}} &
\CD \\
\CC\boxtimes_{\FZ(\CC)} \Fun_{\CC|\CC}(\CC, \CC) \boxtimes_{\FZ(\CC)} g^{(0)} \ar[r]^\simeq &
\CC\boxtimes_{\FZ(\CC)} \Fun_{\CC|\CC}(\CC, \CC \boxtimes_{\FZ(\CC)} g^{(0)}) \ar[u]^{\Id_\CC \boxtimes_{\FZ(\CC)} \Fun_{\CC|\CC}(\CC, g^{(1)})} \ar[r]^-\gamma  &
\CC\boxtimes_{\FZ(\CC)} g^{(0)} \ar[u]^{g^{(1)}} \nonumber
}
\ee
where both $\widetilde{(\underline{g})}^{(1)}$ and $\gamma$ are defined by the equivalence \eqref{eq:C-fun-CD=D}; the commutativity of the left sub-diagram is just the definition of $\phi^{(0)}$; that of the right sub-diagram is obvious. 
Notice that the composition of the arrows in the left column and those in the bottom row is nothing but the identity functor. Therefore, diagram (\ref{map:phi-1+prop-2}) is commutative up to a canonical natural isomorphism, denoted by $\phi^{(1)}$.

Using the explicit formula $f\boxtimes_{\FZ(\CD)} g \mapsto f\boxtimes_\CD g$ in Theorem\,\ref{thm:M-N}, it is routine to check that the isomorphism $g\simeq\widetilde{(\underline{g})}$ given by the pair $(\phi^{(0)},\phi^{(1)})$ is a natural isomorphism.
\end{proof}

\begin{rem}
It is useful to know how much of a $\bk$-linear monoidal functor $f$ is determined by $f^{(0)}$. One can introduce an equivalence relation between monoidal functors in $\Fun^\otimes(\CC, \CD)$: $f\sim g$ if there is a $\bk$-linear monoidal auto-equivalence $h:\CD \to \CD$ such that $f\simeq h\circ g$. We denote the equivalence class of $f$ by $[f]$ and the equivalence class of $f^{(0)}$ by $[f^{(0)}]$. Then the map $[f] \mapsto [f^{(0)}]$ is a bijection.
\end{rem}

\subsection{Fully-faithfulness of the center functor}   \label{sec:fully-faithful}

In this subsection, we assume $k$ is an algebraically closed field of characteristic zero.

\begin{defn}
We say that a finite category $\CC$ over $k$ is {\it semisimple} if $\CC\simeq\RMod_A(\bk)$ for some finite-dimensional semisimple $k$-algebra $A$.
A {\it multi-fusion category} is a semisimple multi-tensor category. A {\it fusion category} is a multi-fusion category with a simple tensor unit.
\end{defn}

We need the following fundamental result.

\begin{thm}\cite{eno2005} \label{thm:eno-semisimple}
Let $\CC$ be a multi-fusion category and $\CM,\CN$ semisimple left $\CC$-modules. Then the category $\Fun_\CC(\CM,\CN)$ is semisimple. In particular, $\Fun_\CC(\CM,\CM)$ is a multi-fusion category.
\end{thm}

\void{
\begin{rem}
The existence of a closed multi-fusion $\CC$-$\CD$-bimodule $\CM$ requires that both $\CC$ and $\CD$ are closed (or nondegenerate).
\end{rem}
}

\begin{defn}
We say that a braided fusion category $\CC$ is {\it nondegenerate}, if the evident braided monoidal functor $\overline\CC\boxtimes\CC\to\FZ(\CC)$ is an equivalence; that is, the monoidal $\CC$-$\CC$-bimodule $\CC$ is closed.
\end{defn}

\begin{rem}
We refer readers to \cite{dgno} for equivalent conditions of the nondegeneracy of a braided fusion category.
\end{rem}


\begin{defn} \label{def:mf-bimodule}
Let $\CC, \CD$ be braided multi-fusion categories and $\CM$ a monoidal $\CC$-$\CD$-bimodule. We say that $\CM$ is a {\it multi-fusion $\CC$-$\CD$-bimodule} if $\CM$ is also a multi-fusion category.
\end{defn}

\begin{thm} \label{thm:closed-bimodule}
Let $\CC, \CD, \CE$ be nondegenerate braided fusion categories and ${}_\CC\CM_\CD$, ${}_\CD\CN_\CE$ closed multi-fusion bimodules. Then $\CM\boxtimes_\CD \CN$ is a closed multi-fusion $\CC$-$\CE$-bimodule.
\end{thm}

\begin{proof}
Note that we have $\CM\boxtimes_\CD\CN \simeq \CM\boxtimes_{\overline\CC\boxtimes\CD} (\CC^\rev\boxtimes\CN)$, and $\CC^\rev\boxtimes\CN$ is a closed multi-fusion $\overline\CC\boxtimes\CD$-$\overline\CC\boxtimes\CE$-bimodule. Therefore,
it is enough to prove the special case where $\CC=\bk$ and $\CD=\FZ(\CM)$.

Let $\CP = \CM\boxtimes_{\FZ(\CM)}\CN$. By Corollary \ref{cor:rigidity} and Theorem \ref{thm:eno-semisimple}, $\CP$ is rigid, and therefore $\CP$ is a multi-fusion category. Since $\CD=\FZ(\CM)$ is a braided fusion category, the multi-fusion category $\CM$ is indecomposable. Invoking Theorem \ref{thm:fun=ph-mor}, we derive an equivalence of monoidal $\FZ(\CM)$-$\FZ(\CP)$-bimodules $\CN \simeq \Fun_{\CM|\CM}(\CM,\CP) = \Fun_{\CM|\CP}(\CP,\CP)$. Then, by Corollary 3.35 in \cite{eo} (see also \cite{schauenburg}), the canonical braided monoidal functor $\FZ(\CM^\rev\boxtimes\CP)\to\FZ(\CN)$ is an equivalence.
Then, from the assumption $\FZ(\CN)\simeq\overline{\FZ(\CM)}\boxtimes\CE$, we conclude that the canonical functor $\CE\to\FZ(\CP)$ is an equivalence.
\end{proof}

We introduce two categories $\mfus$ and $\bfus$ as follows:
\bnu

\item The category $\mfus$ of indecomposable multi-fusion categories over $k$ with the equivalence classes of nonzero semisimple bimodules as morphisms.

\item The category $\bfus$ of nondegenerate braided fusion categories over $k$ with the equivalence classes of closed multi-fusion bimodules as morphisms.

\enu
Note that $\mfus$ and $\bfus$ are well-defined due to Theorem \ref{thm:eno-semisimple} and Theorem \ref{thm:closed-bimodule}.
Both are symmetric monoidal categories under $\boxtimes$.

\begin{thm} \label{thm:fully-faithful}
The center functor from Theorem \ref{thm:functorial} restricts to a fully faithful functor $\FZ: \mfus \to \bfus$.
\end{thm}
\begin{proof}
By Corollary 3.9 in \cite{dgno} (see also \cite{mueger2}), the center of a fusion category is nondegenerate. Thus by Theorem \ref{thm:structure}(7), the center of an indecomposable multi-fusion category is also nondegenerate. This shows that the functor $\FZ$ is well-defined on objects. Let $\CC,\CD$ be indecomposable multi-fusion categories and $\CM$ a nonzero semisimple $\CC$-$\CD$-bimodule. Then the canonical braided monoidal functor $\FZ(\CC^\rev\boxtimes\CD) \to \FZ(\Fun_{\CC|\CD}(\CM,\CM))$ is an equivalence by Corollary 3.35 in \cite{eo}. This shows that the functor $\FZ$ is well-defined on morphisms.

We complete the proof by showing the fully-faithfulness of $\FZ$ in three claims:

{\it Claim 1}. The linear map $\FZ: \Hom_{\mfus}(\bk,\bk) \to \Hom_{\bfus}(\bk, \bk)$ is bijective. In fact, if $\CE$ is a closed multi-fusion $\bk$-$\bk$-bimodule, i.e. $\FZ(\CE)\simeq\bk$, then Corollary \ref{cor:matrix} implies that $\CE\simeq\Fun_{\bk|\bk}(\CM,\CM)$ for some nonzero semisimple category $\CM$. This shows that $\FZ$ is surjective. The injectivity of $\FZ$ is obvious.

{\it Claim 2}. The linear map $\FZ: \Hom_{\mfus}(\CC,\bk) \to \Hom_{\bfus}(\FZ(\CC),\bk)$ is bijective for an indecomposable multi-fusion category $\CC$. Let $\CE$ be a closed multi-fusion $\FZ(\CC)$-$\bk$-bimodule. Note that $\CC\simeq\Fun_{\bk|\CC}(\CC,\CC)$ represents a morphism $\bk\to\FZ(\CC)$ in $\bfus$. So, $\CC\boxtimes_{\FZ(\CC)}\CE$ represents a morphism $\bk\to\bk$ in $\bfus$. Claim 1 then implies that $\CC\boxtimes_{\FZ(\CC)}\CE \simeq \Fun_{\bk|\bk}(\CM,\CM)$ for some nonzero semisimple category $\CM$. Invoking Theorem \ref{thm:fun=ph-mor}, we obtain a $\bk$-linear monoidal functor $\CC\to\Fun_{\bk|\bk}(\CM,\CM)$.

Conversely, let $\CM$ be a nonzero semisimple $\CC$-$\bk$-bimodule, regarded as a $\bk$-linear monoidal functor $\CC\to\Fun_{\bk|\bk}(\CM,\CM)$. Theorem \ref{thm:fun=ph-mor} says that there is a monoidal $\FZ(\CC)$-$\bk$-bimodule $\CE$ and an equivalence $\CC\boxtimes_{\FZ(\CC)}\CE \simeq \Fun_{\bk|\bk}(\CM,\CM)$. Moreover, we may take $\CE = \Fun_{\CC|\CC}(\CC,\Fun_{\bk|\bk}(\CM,\CM))$, which is nothing but $\Fun_{\CC|\bk}(\CM,\CM)$, the image of $\CM$ under the center functor.

Theorem \ref{thm:fun=ph-mor} states that these two constructions are inverse to each other, so we conclude the claim.

{\it Claim 3}. The linear map $\FZ: \Hom_{\mfus}(\CC,\CD) \to \Hom_{\bfus}(\FZ(\CC),\FZ(\CD))$ is bijective for indecomposable multi-fusion categories $\CC,\CD$. Actually, this map is equivalent to the map $\FZ: \Hom_{\mfus}(\CC\boxtimes\CD^\rev,\bk) \to \Hom_{\bfus}(\FZ(\CC\boxtimes\CD^\rev),\bk)$, which is bijective by Claim 2.
\end{proof}

\begin{rem}
For a fusion category $\CC$, it was known that there is a bijection between closed fusion modules over $\FZ(\CC)$ and indecomposable semisimple $\CC$-modules \cite{eno2008, eno2009,dmno}. We 
give a new proof of this result and provide a conceptual framework to understand it. 
\end{rem}


Note that every multi-fusion category is a direct sum of indecomposable ones. Combining Theorem \ref{thm:fully-faithful} with Proposition \ref{prop:inv-bimod-center}, we obtain the following corollary.
\begin{cor}[\cite{eno2008}] \label{cor:eno2008}
Two multi-fusion categories $\CC$ and $\CD$ are Morita equivalent if and only if $\FZ(\CC) \simeq \FZ(\CD)$ as braided multi-fusion categories.
\end{cor}



Given a multi-fusion category $\CC$, we denote the group of the equivalence classes of semisimple invertible $\CC$-$\CC$-bimodules by $\mathrm{BrPic}(\CC)$, and denote the group of the isomorphism classes of braided auto-equivalences of $\FZ(\CC)$ by $\mathrm{Aut}^{br}(\FZ(\CC))$. The following result is a consequence of Proposition\,\ref{prop:inv-bimod-center}, Example\,\ref{exam:invertible=braided-auto} and Theorem\,\ref{thm:fully-faithful}.

\begin{cor}[\cite{eno2009}] \label{cor:eno2009}
Let $\CC$ be an indecomposable multi-fusion category. We have $\mathrm{BrPic}(\CC) \simeq \mathrm{Aut}^{br}(\FZ(\CC))$ as groups.
\end{cor}

\begin{rem}{\rm
The only-if part of Corollary\,\ref{cor:eno2008} was proved by M\"{u}ger \cite{mueger0}. Etingof, Nikshych and Ostrik proved above two results in \cite{eno2008,eno2009} for fusion categories. It is straightforward to generalize their results to multi-fusion categories. Our proof is different. Corollary \ref{cor:eno2009} also holds for tensor categories \cite{dn}.
}
\end{rem}

\begin{rem}
The physical meaning of Corollary\,\ref{cor:eno2008} and Corollary\,\ref{cor:eno2009} was explained in \cite{kitaev-kong,kong-icmp}, that of Theorem\,\ref{thm:fully-faithful} was explained in \cite{kong-wen-zheng}.
\end{rem}


\end{document}